\documentclass[12pt]{article}
\usepackage{epsfig}
\usepackage{authblk}
\usepackage{epstopdf}
\usepackage{graphics}
\usepackage{amsmath}
\usepackage{amsthm}
\usepackage{amssymb}
\usepackage{array}
\usepackage{subfig} 
\usepackage[font={footnotesize}]{caption}
\usepackage{relsize}
\usepackage[nottoc]{tocbibind}
\usepackage{hyperref}
\usepackage{float}

\usepackage{chngcntr}
\usepackage{wrapfig}

\counterwithin*{equation}{section}
  
\newtheorem{thm}{Theorem}[section] 
\newtheorem{lem}[thm]{Lemma}  
\newtheorem{cor}[thm]{Corollary} 
\newtheorem{prop}[thm]{Proposition}  

\newtheorem{hyp}[thm]{Hypothesis}
\newtheorem{claim}[thm]{Claim}
\newtheorem{defn}[thm]{Definition}

\newtheoremstyle{named}{}{}{\itshape}{}{\bfseries}{.}{.5em}{\thmnote{#3's }#1}
\theoremstyle{named}

\theoremstyle{definition}
\newtheorem{rmk}{Remark}

\title{Stability of Infinite Systems of Coupled Oscillators Via Random Walks on Weighted Graphs}

\author{Jason J. Bramburger\\ Division of Applied Mathematics\\ Brown University \\ Providence, Rhode Island 02912\\ USA}

\date{} 

\begin{document} 

\maketitle

\abstract{Weakly coupled oscillators are used throughout the physical sciences, particularly in mathematical neuroscience to describe the interaction of neurons in the brain. Systems of weakly coupled oscillators have a well-known decomposition to a canonical phase model which forms the basis of our investigation in this work. Particularly, our interest lies in examining the stability of synchronous (phase-locked) solutions to this phase system: solutions with phases having the same temporal frequency but differ through time-independent phase-lags. The main stability result of this work comes from adapting a series of investigations into random walks on infinite weighted graphs. We provide an interesting link between the seemingly unrelated areas of coupled oscillators and random walks to obtain algebraic decay rates of small perturbations off the phase-locked solutions under some minor technical assumptions. We also provide some interesting and motivating examples that demonstrate the stability of phase-locked solutions, particularly that of a rotating wave solution arising in a well studied paradigm in the theory of coupled oscillators. }

\section{Introduction} \label{sec:Introduction} 

Oscillatory behaviour is well-known to arise in many areas of biology, particularly in the study of neural networks. Through chemical and/or electrical synapses, neurons can fall into complex synchronous periodic behaviour patterns governing a wide variety of cognitive tasks $\cite{Buschman,CarmichaelChesselet,JutrasBuffalo}$ and potentially aiding in memory formation $\cite{MohnsBlumberg}$. An important approach to studying the dynamic behaviour of neural activity mathematically has been through the study of weakly coupled oscillators $\cite{ErmentroutChow,WeaklyConnectedBook}$. Hence, from a theoretical perspective, modern dynamical systems theory can be used to greatly enhance our understanding of synchronous patterns in neural networks.      

One of the most important characteristics of studying weakly coupled oscillators is that each oscillator can be reduced through a process of averaging to a single phase variable lying on the circle $S^1$ under minor technical assumptions \cite{Corinto,DeVille,ErmentroutAverage,WeaklyConnectedBook,Udeigwe}. It is exactly these so-called `phase models' that have allowed researchers to identify patterns of synchrony arising in the study of coupled oscillators. To date there have been many nontrivial stable synchrony patterns that have been identified, including traveling waves, target patterns and rotating waves $\cite{ErmentroutStability,ErmentroutRen,Paullet,ErmentroutSpiral,Udeigwe}$. Some authors have even postulated that such patterns of synchrony in neurons could be related to simple geometric structures perceived during visual hallucinations $\cite{ErmentroutCowan,Tass}$.

In this paper we will work to further the current understanding of the stability in coupled phase models by inspecting an infinite system of ordinary differential equations. We will consider a countably infinite index set $V$ along with the system of ordinary differential equations indexed by the elements of $V$ given by
\begin{equation} \label{PhaseSystem}
	\dot{\theta}_v(t) = \omega_v + \sum_{v'\in N(v)} H(\theta_{v'}(t) - \theta_v(t)),
\end{equation}
for each $v \in V$. Here $\dot{x} = dx/dt$, $\theta_v: \mathbb{R}^+ \to S^1$ and $H:S^1 \to S^1$ is (at least) twice differentiable and a $2\pi$-periodic function describing the interaction between oscillators indexed by the elements of $N(v) \subseteq V\setminus\{v\}$ for all $v \in V$. One thinks of the elements in $N(v)$ as belonging to a neighbourhood of the element $v$. More precisely, we will say that if $v' \in N(v)$ the state $\theta_{v'}$ {\em influences} the state $\theta_{v}$. Throughout this manuscript we will assume the influence topology to be symmetric, which formally means that $v \in N(v')$ if and only if $v' \in N(v)$ for all $v,v'\in V$. The $\omega_v$ are taken to be real-valued constants representing intrinsic differences in the oscillators and/or external inputs. Our interest in system $(\ref{PhaseSystem})$ is to describe sufficient conditions for the stability of synchronized patterns of oscillation arising as solutions to this model. 

The differential equation $(\ref{PhaseSystem})$ can be interpreted as a generalization of the celebrated Kuramoto model which has widely been applied in mathematical neuroscience $\cite{Cumin, Kuramoto}$. Particularly, we will see through the applications of our work here that the motivating examples of interaction functions $H$ are given by sinusoidal functions, as in the Kuramoto model. Furthermore, this work refrains from inspecting the case when $V$ has a finite number of elements since this was taken care of in a previous study by Ermentrout $\cite{ErmentroutStability}$. It was exactly this study that made explicit the link between spectral graph theory and the stability of coupled oscillators, showing that under mild conditions one may interpret the linearization about a synchronous solution as corresponding to a weighted graph. Many investigations have used these results to show that finite systems of coupled oscillators exhibit stable solutions for an arbitrary finite number of oscillators, but most fail to mention how stability changes as the number of oscillators increases without bound. What could be expected is that the spectral gap in this case shrinks, causing eigenvalues to converge onto the imaginary axis as the number of oscillators tends to infinity. It is exactly this type of phenomena which has been a recent area of investigation for partial differential equations $\cite{SandstedeScheel}$. Studies have shown that the truncation from infinite to finite can help to stabilize otherwise unstable solutions, and that faint remnants of the instabilities reveal themselves for large finite truncations. Therefore, in this work we wish to detail sufficient conditions for stability in infinite dimensions, allowing one to address all of the previously stated concerns as it pertains to a well-studied paradigm in mathematical neuroscience.    

Typical investigations into coupled phase models focus on one or both of the following aspects: the interaction between oscillators (different $H$ functions) and the effect of coupling topologies (different $V$ and $N(v)$). The work presented here focusses more on the latter since we will follow similar method's utilized in the finite dimensional setting to understand how the coupling topologies can endow an underlying graph theoretical framework to help understand the stability of synchronous solutions. We will see that this graph theoretic framework can reward networks with more complex connections with faster decay rates of slight perturbations. This paper will use some of the well-developed theory of random walks on infinite weighted graphs to obtain algebraic (as opposed to exponential) decay rates of small perturbations off synchronous solutions to our phase model $(\ref{PhaseSystem})$. In fact, the reason we have chosen $V$ to represent the index set for our differential equation is that we will see in the coming sections it can be shown to correspond to the vertices of an infinite weighted graph. In doing so we provide a fascinating link between two seemingly unrelated areas of mathematics and introduce a new application for some very theoretical mathematical tools.  

System $(\ref{PhaseSystem})$ generalizes the Kuramoto-type coupled phase model studied in $\cite{MyWork}$ given by
\begin{equation}\label{PhaseSystem2}
	\dot{\theta}_{i,j} = \omega + \sin(\theta_{i+1,j} - \theta_{i,j}) + \sin(\theta_{i-1,j} - \theta_{i,j}) + \sin(\theta_{i,j+1} - \theta_{i,j}) + \sin(\theta_{i,j-1} - \theta_{i,j}),
\end{equation}
for each $(i,j)\in\mathbb{Z}^2$. In particular, system (\ref{PhaseSystem2}) was shown to possess a rotating wave solution, which as previously mentioned, pertains to important biological phenomena. In this work we will extend the investigation of $\cite{MyWork}$ by demonstrating that not only the rotating wave solution is stable, but there are infinitely many stable synchronous states to the model (\ref{PhaseSystem2}) which can be analyzed using the methods of this paper. By using a particular phase model as a case study, we aid the reader in better understanding the techniques of this paper and further some known mathematical results.    

This paper is organized as follows. In Section $\ref{sec:Perturbation}$ we lay out the basic mathematical framework that allows us to study the stability of synchronous solutions. In Section $\ref{sec:Graphs}$ we provide a brief overview of the relevant results and techniques from graph theory, including an investigation of random walks on infinite weighted graphs in Section $\ref{subsec:RandomWalks}$. These results on graphs are then used to formulate the main stability result of this paper, Theorem~\ref{thm:Stability}, which is presented in Section $\ref{sec:Stability}$. Theorem~\ref{thm:Stability} shows that if the graph defined by the linearization about the synchronous solution has dimension at least two, then we may obtain nonlinear stability of said synchronous solution. The proof of Theorem~\ref{thm:Stability} is left to Section $\ref{sec:StabilityProof}$ where a typical Picard iteration method is employed in conjunction with the results from infinite weighted graphs. In Section $\ref{sec:Applications}$ we present a thorough analysis of some stable states in a particular coupled phase model, demonstrating that there are infinitely many stable synchronous states in this model. Finally in Section $\ref{sec:Discussion}$ we present a brief discussion of the results in this work, particularly regarding the shortcomings of the main stability result regarding one dimensional graphs, along with some future directions for inquiry.

\section{The Perturbation System} \label{sec:Perturbation} 

Our interest in this work lies in investigating synchronous solutions to $(\ref{PhaseSystem})$. Such solutions are all oscillating with identical frequency and are typically referred to as phase-locked. Solutions of this type take the form 
\begin{equation} \label{PhaseAnsatz}
	\theta_v(t) = \Omega t + \bar{\theta}_v,
\end{equation}  
where $\bar{\theta}_v$ is a time-independent phase-lag for all $v \in V$ and the elements $\theta_v(t)$ are identically oscillating with period $2\pi/\Omega$, for some $\Omega \in \mathbb{R}$. The ansatz $(\ref{PhaseAnsatz})$ reduces $(\ref{PhaseSystem})$ to solving the infinite system  
\begin{equation} \label{PhaseLagsEqn}
	(\omega_v - \Omega) + \sum_{v' \in N(v)} H(\bar{\theta}_{v'} - \bar{\theta}_v) = 0
\end{equation}
for the phase-lags $\bar{\theta} = \{\bar{\theta}_v\}_{v \in V}$ and the frequency $\Omega$. 

Assuming we have a solution to $(\ref{PhaseLagsEqn})$, our interest turns to applying a slight perturbation to the ansatz $(\ref{PhaseAnsatz})$, written
\begin{equation} \label{PerturbedAnsatz}
	\theta_v(t) = \Omega t + \bar{\theta}_v + \psi_v(t),
\end{equation}
and inspecting conditions to which $\psi_v(t) \to 0$ as $t \to \infty$ for all $v \in V$. Notice that the perturbed ansatz $(\ref{PerturbedAnsatz})$ leads to the system of differential equations
\begin{equation} \label{PerturbationSystem}
	\dot{\psi}_v = (\omega_v - \Omega) + \sum_{v' \in N(v)} H(\bar{\theta}_{v'} + \psi_{v'} - \bar{\theta}_{v} - \psi_{v}), \ \ \ \ \ v \in V,
\end{equation} 
which from $(\ref{PhaseLagsEqn})$ has a steady-state solution given by $\psi = 0$. We have also suppressed the dependence of the variables in $(\ref{PerturbationSystem})$ on $t$ for the ease of notation. System $(\ref{PerturbationSystem})$ forms the basis for our investigation in this work. 

As is well-known from the study of finite-dimensional ordinary differential equations, linearizing about a steady-state solution can provide great insight into the stability of this solution. In the case of infinite-dimensional ordinary differential equations one may follow this line of inquiry and attempt similar techniques, but there are some subtleties which only reveal themselves in infinite dimensions. Notice that linearizing $(\ref{PerturbationSystem})$ about its steady-state solution $\psi = 0$ results in the linear operator, $L_{\bar{\theta}}$, acting upon the sequences $x = \{x_v\}_{v \in V}$ by
\begin{equation} \label{Linearization}
	[L_{\bar{\theta}}x]_v = \sum_{v' \in N(v)} H'(\bar{\theta}_{v'} - \bar{\theta}_{v})(x_{v'} - x_v),
\end{equation} 
for each $v \in V$. Operators of the form (\ref{Linearization}) leave the possibility that a continuum of spectral values intersects the imaginary axis, and therefore typical exponential stability results relying on the existence of a spectral gap cannot be inferred about any phase-locked solution to $(\ref{PhaseSystem})$.

Let us illustrate with a prototypical example so that the reader can better comprehend the situation we wish to investigate here. Consider $V = \mathbb{Z}^2$ and an associated linear operator in the form of (\ref{Linearization}), acting upon the sequences $x = \{x_{i,j}\}_{(i,j) \in \mathbb{Z}}$, given by 
\begin{equation} \label{LinearizationEx}
	[Lx]_{i,j} = (x_{i+1,j} - x_{i,j}) + (x_{i-1,j} - x_{i,j}) + (x_{i,j+1} - x_{i,j}) + (x_{i,j-1} - x_{i,j}),	
\end{equation}
which is referred to as the two-dimensional discrete Laplacian operator. It is well-known that when posed upon the spaces $\ell^2(\mathbb{Z}^2)$ or $\ell^\infty(\mathbb{Z}^2)$, defined below in (\ref{ellpSpace}) and (\ref{ellInfty}), respectively, the spectrum of (\ref{LinearizationEx}) is exactly the subset of the complex plane given by the real interval $[-4,0]$ \cite{Cahn}. Hence, we see that the spectrum is both uncountable and intersects the imaginary axis of the complex plane, therefore implying that exponential dichotomies based upon the linearization (\ref{LinearizationEx}) cannot be inferred, at least when working with the Banach spaces $\ell^2(\mathbb{Z}^2)$ or $\ell^\infty(\mathbb{Z}^2)$. In Section~\ref{subsec:Trivial} we will return to this example to see that stability can in fact be inferred based upon the linearization (\ref{LinearizationEx}) using the stability theorem presented in this work. 

Moving back to the general situation of (\ref{Linearization}), when $H'$ is uniformly bounded above, a simple argument following similarly to that which is laid out in Proposition $5.2$ of $\cite{MyWork}$ shows that $L_{\bar{\theta}}$ is not a Fredholm operator on $\ell^\infty(V)$, and again we cannot obtain exponential dichotomies based upon these linear operators. Therefore, it is such situations that necessitates the work on stability of phase-locked solutions which is presented in this manuscript. We will overcome this failure to produce exponential dichotomies and obtain algebraic dichotomies instead. Prior to providing the main stability result of this paper, we require a brief overview of the relevant graph theoretic techniques which we will be applying in this work.

\section{Denumerable Graph Networks} \label{sec:Graphs} 

Throughout the following subsections we will provide the necessary definitions and results to obtain our main result Theorem~\ref{thm:Stability}. We will use similar nomenclature and notation to that of Delmotte $\cite{Delmotte}$, which appears to now be quite standard. Another excellent source which summarizes much of the relevant results and more is Telcs' textbook $\cite{TelcsBook}$. The majority of this section comes as a brief literature review to bring the reader up to date on some of the relevant results pertaining to our work in this manuscript. We will also provide some brief lemmas and corollaries to extrapolate the results and make them more easily applicable to our work here.

\subsection{Preliminary Definitions} 

We consider a graph $G = (V,E)$ with a countably infinite collection of vertices, $V$, and a set of unoriented edges between these vertices, $E$, with the property that at most one edge can connect two vertices. We refer to a loop as an edge which initiates and terminates at the same vertex. If there exists an edge $e\in E$ connecting the two vertices $v,v'\in V$ then we write $v\sim v'$, and since the edges are unoriented this relation is naturally symmetric in that $v' \sim v$ as well. In this way we may equivalently consider the set of edges $E$ as a subset of the product $V \times V$. A graph is called connected if for any two vertices $v,v'\in V$ there exists a finite sequence of vertices in $V$, $\{v_1,v_2, \dots , v_n\}$, such that $v\sim v_1$, $v_1\sim v_2$, $\dots$, $v_n \sim v'$. We will only consider connected graphs for the duration of this work. 

We will also consider a weight function on the edges between vertices, written $w:V\times V \to [0,\infty)$, such that for all $v,v'\in V$ we have $w(v,v') = w(v',v)$ and $w(v,v') > 0$ if and only if $v \sim v'$. This then leads to the notion of a weighted graph, written as the triple $G = (V,E,w)$. The weight function also naturally extends to the notion of the measure (or sometimes weight) of a vertex, $m:V \to [0,\infty]$, defined by
\begin{equation} \label{VertexMeasure}
	m(v) := \sum_{v'\in V} w(v,v') = \sum_{v\sim v'} w(v,v').
\end{equation} 
Throughout this work we will only consider graphs and weight functions such that $m(v) < \infty$ for all $v \in V$. The weight function then leads to the definition of an operator acting on the graph. 

\begin{defn} \label{def:Laplacian} 
	For any function $f:V \to \mathbb{C}$ acting on the vertices of $G = (V,E,w)$, we define the {\bf graph Laplacian} to be the linear operator $L_G$ acting on these functions by
	\begin{equation} \label{NormLap}
		L_Gf(v) = \frac{1}{m(v)}\sum_{v\sim v'} w(v,v')(f(v') - f(v)). 
	\end{equation} 
\end{defn}  

Natural spatial settings for the graph Laplacian operator of Definition $\ref{def:Laplacian}$ are the sequence spaces 
\begin{equation} \label{ellpSpace}
	\ell^p(V) = \{f:V \to \mathbb{C}\ |\ \sum_{v\in V} |f(v)|^p < \infty\},	
\end{equation}  
for any $p\in [1,\infty)$. The vector space $\ell^p(V)$ become a Banach space when equipped with the norm
\begin{equation} 
	\|f\|_p := \bigg(\sum_{v\in V} |f(v)|^p\bigg)^{\frac{1}{p}}.
\end{equation} 
We may also consider the Banach space $\ell^\infty(V)$, the vector space of all uniformly bounded functions $f:V\to \mathbb{C}$ with norm given by
\begin{equation} \label{ellInfty}
	\|f\|_\infty := \sup_{v\in V} |f(v)|.
\end{equation}
One often writes the elements of the sequence spaces in the alternate form as sequences indexed by the elements of $V$, $f = \{f_v\}_{v\in V}$, where $f_v := f(v)$. Also, it should be noted that these definitions extend to any countable index set $V$, independent of a respective graph. 

Throughout this work we will also consider bounded linear operators acting between these sequence spaces. That is, consider a bounded linear operator $T:\ell^p(V) \to \ell^q(V)$ for some $1 \leq p,q \leq \infty$. We denote the norm of this operator as 
\begin{equation} \label{ellpOpNorm}
	\|T\|_{p \to q} := \sup_{0 \neq f \in \ell^p(V)} \frac{\|Tf\|_q}{\|f\|_p}.
\end{equation} 
One should note that there are many different, but equivalent, versions of this norm which one may work with. We now wish to provide sufficient conditions for which a graph Laplacian defines a bounded linear operator acting from $\ell^p(V)$ back into itself, for every $p \in [1,\infty]$. Before doing so, we remark that for a vertex $v \in V$ we define the degree to be
\begin{equation}
	{\rm Deg}(v) = \#\{v'\in V |\ v\sim v'\},
\end{equation}
where $\#\{\cdot\}$ represents the cardinality of the set. We use this definition to provide the following results. 

\begin{lem} \label{lem:NormLaplacian}
	Let $G = (V,E,w)$ be a weighted graph. If there exists finite $D > 0$ such that ${\rm Deg}(v) \leq D$ for all $v \in V$ then the graph Laplacian $(\ref{NormLap})$ defines a bounded linear operator on $\ell^p(V)$ for all $p\in [1,\infty]$.	
\end{lem}

\begin{proof}
	Prior to working with a specific sequence space let us begin by noticing that from the definition of the measure $(\ref{VertexMeasure})$ we have
	\begin{equation}
		\frac{w(v,v')}{m(v)} \leq \sum_{v\sim v'} \frac{w(v,v')}{m(v)} = 1,
	\end{equation} 
	for all $v,v' \in V$. Hence,
	\begin{equation} \label{LNormInequality}
		\begin{split}
			|L_Gf(v)| &\leq \frac{1}{m(v)}\sum_{v\sim v'} w(v,v')(|f(v')| + |f(v)|) \\
			&= |f(v)| + \sum_{v\sim v'} \frac{w(v,v')}{m(v)}|f(v')|, 
		\end{split}	
	\end{equation}
	for all $v \in V$. One may apply H\"older's Inequality to find 
	\begin{equation}
		\sum_{v\sim v'} \frac{w(v,v')}{m(v)}|f(v')| \leq \bigg(\sup_{v' \in V} \frac{w(v,v')}{m(v)}\bigg) \bigg(\sum_{v\sim v'} |f(v')|\bigg) \leq \sum_{v\sim v'} |f(v')|,	
	\end{equation}
	for all $v \in V$. Combining this with $(\ref{LNormInequality})$ gives
	\begin{equation} \label{LNormInequality2}
		|L_Gf(v)| \leq |f(v)| + \sum_{v\sim v'} |f(v')|. 	
	\end{equation}
	We now use $(\ref{LNormInequality2})$ to demonstrate boundedness of the operator $L$ on the sequence spaces, starting with $\ell^1(V)$ and $\ell^\infty(V)$.
	
	$\underline{L_G:\ell^1(V) \to \ell^1(V)}$: Let $f \in \ell^1(V)$. Summing $|L_Gf(v)|$ over all $v \in V$ and applying $(\ref{LNormInequality2})$ gives
	\begin{equation}
		\begin{split}
			\|L_Gf\|_1 &= \sum_{v \in V} |Lf(v)| \\
				 &\leq  \sum_{v \in V} |f(v)| + \sum_{v \in V} \sum_{v\sim v'} |f(v')| \\ 
				 &\leq \|f\|_1 + D\|f\|_1 \\
				 & = (D+1)\|f\|_1.
		\end{split} 
	\end{equation}
	This therefore shows that $\|L_G\|_{1\to 1} \leq D+1$.
	
	$\underline{L_G:\ell^\infty(V) \to \ell^\infty(V)}$: Let $f \in \ell^\infty(V)$. Taking the supremum of $|L_Gf(v)|$ over all $v \in V$ and applying $(\ref{LNormInequality2})$ gives
	\begin{equation}
		\begin{split}
			\|L_Gf\|_\infty &= \sup_{v \in V} |Lf(v)| \\
				&\leq \sup_{v \in V} |f(v)| + \sum_{v\sim v'} |f(v')| \\
				&\leq (D+1) \|f\|_\infty.
		\end{split}
	\end{equation}
	Hence, $\|L_G\|_{\infty \to \infty} \leq D+1$. 
	 
	Now that we have shown that $L_G$ is bounded on both $\ell^1$ and $\ell^\infty$, standard interpolation result over the $\ell^p(V)$ spaces imply that $L_G:\ell^p(V) \to \ell^p(V)$ are uniformly bounded for all $1 < p < \infty$ (see for example Exercise $12$ of $\S 2.6$ from $\cite{InfiniteMatrices}$). This completes the proof.
\end{proof}

\subsection{Graphs as Metric Measure Spaces} 

In this section we extend some of the notions introduced about graphs and provide the necessary nomenclature to introduce random walks on graphs. We saw that for a weighted graph, $G = (V,E,w)$, we define the measure of a vertex as in $(\ref{VertexMeasure})$. This notion extends to the volume of a subset, $V_0 \subset V$, by defining 
\begin{equation}
	{\rm Vol}(V_0) := \sum_{v \in V_0} m(v).
\end{equation} 
Hence, under this definition of volume, a weighted graph naturally becomes a measure space on the $\sigma$-algebra given by the power set of $V$.  

Graphs also have a natural underlying metric, $\rho$, given by the function which returns the smallest number of edges to produce a path between two vertices $v,v' \in V$. Note that since $G$ is assumed connected, the distance function is well-defined. This metric allows for the consideration of a ball of radius $r \geq 0$ centred at the vertex $v \in V$, denoted by
\begin{equation}
	B(v,r) := \{v'\ |\ \rho(v,v') \leq r\}.
\end{equation} 
In this work we will write ${\rm Vol}(v,r)$ to denote ${\rm Vol}(B(v,r))$. The combination of the graph metric and the vertex measure allows one to interpret a weighted graph as a {\em metric-measure space}.  

We provide a series of definitions to further our understanding of graphs as metric-measure spaces. 

\begin{defn} \label{def:VolumeGrowth} 
	The weighted graph $G = (V,E,w)$ satisfies a {\bf uniform polynomial volume growth} condition of order $d$, abbreviated VG(d), if there exists $d > 0$ and $c_{vol,1},c_{vol,2} > 0$ such that
	\begin{equation}
		c_{vol,1}r^d \leq {\rm Vol}(v,r) \leq c_{vol,2}r^d,
	\end{equation}
	for all $v \in V$ and $r \geq 0$.
\end{defn}

In some cases one may also consider graphs with more general volume growth conditions, but for the purposes of this work we will restrict ourselves to polynomial growth conditions. The characteristic examples of graphs satisfying $VG(d)$ are the integer lattices $\mathbb{Z}^d$ with all edges of weight $1$ such that there exists an edge between two vertices $n,n' \in \mathbb{Z}^d$ if and only if $\|n - n'\|_1 = 1$. This is pointed out on, for example, page $10$ of $\cite{BarlowCoulhonGrigoryan}$. For the duration of this work we will restrict our attention to those graphs satisfying VG(d) with $d\geq 2$. 

\begin{defn} \label{def:Delta} 
	We say $G = (V,E,w)$ satisfies the {\bf local elliptic property}, denoted $\Delta$, if there exists an $\alpha > 0$ such that
	\begin{equation}
		w(v,v') \geq \alpha m(v)
	\end{equation}
	for all $v,v' \in V$ such that $v' \sim v$.
\end{defn}

It appears that the notation of using $\Delta$ to denote this local elliptic property has become commonplace in the literature, and therefore we use this notation for consistency. The following lemma points out an important set of sufficient conditions to satisfy $\Delta$.

\begin{lem}\label{lem:Delta} 
	Let $G = (V,E,w)$ be a weighted graph. If there exists constants $D,w_{min}, w_{max} > 0$ such that 
	\begin{equation}
		w_{min} \leq w(v,v') \leq w_{max},
	\end{equation}
	and ${\rm deg}(v) \leq D$ for all $v \in V$ and $v \sim v'$, then $G$ satisfies $\Delta$. 
\end{lem}

\begin{proof}
	For all $v \in V$ we have
	\begin{equation}
		m(v) = \sum_{v \sim v'} w(v,v') \leq Dw_{max}.
	\end{equation}
	This gives
	\begin{equation}
		\frac{w(v,v')}{m(v)} \geq \frac{w_{min}}{Dw_{max}} > 0.
	\end{equation} 
	Thus, $G = (V,E,w)$ will satisfy $\Delta$ for any $\alpha > 0$ such that $w_{min} \geq \alpha Dw_{max}$.
\end{proof}

\begin{defn} 
	The weighted graph $G = (V,E,w)$ satisfies the {\bf Poincar\'e inequality}, abbreviated PI, if there exists a $C_{PI} > 0$ such that
	\begin{equation}
		\sum_{v \in B(v_0,r)} m(v)|f(v) - f_B(v_0)|^2 \leq C_{PI} r^2\bigg( \sum_{v,v' \in B(v_0,2r)} w(v,v')(f(v) - f(v'))^2\bigg),
	\end{equation}	
	for all functions $f:V \to \mathbb{R}$, all $v_0 \in V$, all $r > 0$, where
	\begin{equation}
		f_B(v_0) = \frac{1}{{\rm Vol}(v_0,r)} \sum_{v\in B(v_0,r)} m(v)f(v).
	\end{equation}
\end{defn}

The Poincar\'e Inequality is certainly the most difficult of the three definitions to work with. In practice it can be quite difficult to confirm whether or not a weighted graph satisfies $PI$, although some methods to obtain this inequality are given in $\cite{PoincareInequalities}$. Next we will introduce an important definition and result from $\cite{RoughIsometry}$ that can aid in determining if a graph satisfies $PI$. 

\begin{defn} \label{def:RoughIsometry} 
	Let $G = (V,E,w)$ and $G' = (V',E',w')$ be two infinite weighted graphs satisfying $\Delta$ with respective graph metrics given by $\rho$ and $\rho'$. A map $T: V \to V'$ is called a {\bf rough isometry} if there exists $a,c > 1$, $b > 0$ and $M > 0$ such that
	\begin{subequations}
		\begin{equation} \label{Rough1} 
				a^{-1}\rho(v_1,v_2) - b \leq \rho'(T(v_1),T(v_2)) \leq a\rho(v_1,v_2) + b, \ \ \ \ \ \forall v_1,v_2\in V, 
		\end{equation}	
		\begin{equation} \label{Rough2} 
			\rho'(T(V),v') \leq M, \ \ \ \ \ \forall v'\in V',
		\end{equation}
		\begin{equation} \label{Rough3} 
			c^{-1}m(v) \leq m'(T(v)) \leq cm(v), \ \ \ \ \ \forall v\in V,	
		\end{equation}
	\end{subequations}
	where $m$ and $m'$ are the vertex measures associated to the graph $G$ and $G'$, respectively. If $T: G \to G'$ is a rough isometry, $G$ and $G'$ are said to be {\bf rough isometric}. 
\end{defn}

\begin{prop}[{\em \cite{RoughIsometry}, \S 5.3, Proposition 5.15(2)}] \label{prop:RoughInvariance} 
	Let $G$ and $G'$ be two infinite weighted graphs satisfying $\Delta$ that are rough isometric. Then if there exists a $d > 0$ such that $G$ satisfies $VG(d)$ and $PI$, then $G'$ satisfies $VG(d)$ and $PI$ as well. 
\end{prop}

Hambly and Kumagai's original statement of Proposition $\ref{prop:RoughInvariance}$ refers to our $PI$ as a {\em weak} Poincar\'e inequality since they sometimes use a stronger inequality in their work. Hambly and Kumagai also originally provide their statement in terms of a more general volume growth condition, but we will work with the weaker version stated here since we are only interested in polynomial volume growth. In fact, using property $(\ref{Rough3})$ one can show that property $VG(d)$ is preserved under rough isometries in a straightforward way.

\subsection{Random Walks on Weighted Graphs} \label{subsec:RandomWalks} 

Our interest here will be in continuous time random walks on the vertices of a weighted graph $G = (V,E,w)$. One can interpret this as being at a single vertex on the graph, and then waiting an exponentially distributed amount of time to move along an edge to another vertex of the graph. Upon arriving at the next vertex, this process begins again by waiting an exponentially distributed amount of time to move along an edge to another vertex of the graph. The weights in this scenario act as preferences to moving along a certain edge; the greater the weight, the greater the preference. More precisely, if we are at the vertex $v \in V$, then the probability we move to $v' \sim v$ is given by $w(v,v')/m(v)$. The important thing to note here is that there are exactly two probabilities involved: the probability of when to move from one vertex to the next and the probability of choosing the vertex to move to. We use the notation $p_t(v,v')$ to denote the probability that after time $t \geq 0$ we have arrived at the vertex $v'$ having started at vertex $v$. By definition we have
\begin{equation} \label{ProbSum}
	\sum_{v'\in V} p_t(v,v') = 1
\end{equation} 
for all $v \in V$. Delmotte $\cite{Delmotte}$ points out that the $p_t(\cdot,\cdot)$ are not necessarily symmetric in their arguments due to the weights on the graph, but it has been shown that
\begin{equation}
	\frac{p_t(v,v')}{m(v')} = \frac{p_t(v',v)}{m(v)}.
\end{equation} 
This has prompted some authors $\cite{BernicotCoulhonFrey,RoughIsometry,Horn}$ to instead study the symmetric transition densities 
\begin{equation} \label{qDensity}
	q_t(v,v') := \frac{p_t(v,v')}{m(v')}
\end{equation}
for all $v,v'\in V$.

Much work has been done to understand the long-time behaviour of the probabilities $p_t(\cdot,\cdot)$, notably the pioneering work of Delmotte $\cite{Delmotte}$. Most applicable to our present situation is that these probabilities are used to understand the solution to the spatially discrete heat equation 
\begin{equation} \label{NormalizedLap}
	\dot{x}_v = \frac{1}{m(v)}\sum_{v' \in V} w(v,v')(x_{v'} - x_v),\ \ \ \ \ v\in V.
\end{equation}
When considering all elements $\{x_v\}_{v\in V}$, the right hand side of $(\ref{NormalizedLap})$ is a graph Laplacian operator, again denoted $L_G$. Then as stated in Theorem $23$ of $\cite{KellerLenz}$, $L_G$ is the infinitesimal generator of the semigroup $P_t = e^{L_Gt}$. For an initial condition $x_0 = \{x_{v,0}\}_{v \in V}$ the fundamental solution to $(\ref{NormalizedLap})$ with this initial condition is given by
\begin{equation} \label{HeatSoln}
	x_v(t) = [P_tx_0]_v = \sum_{v' \in V} p_t(v,v')x_{v',0} 
\end{equation}   
for each $v \in V$, thus showing the connection between the probabilities $p_t(\cdot,\cdot)$ and the semigroup $P_t$. The fact that $(\ref{HeatSoln})$ solves $(\ref{NormalizedLap})$ was pointed out by Delmotte, and other sources include, but are not limited to, $\cite{Horn,Weber}$ for continuous time transitions and $\cite{Grigoryan1,Grigoryan2}$ for discrete time transitions. One also can see using the identity $(\ref{qDensity})$, the solution $(\ref{HeatSoln})$ now can be interpreted as the Lebesgue integral
\begin{equation}
	[P_tx_0]_v = \sum_{v' \in V} q_t(v,v')x_{v',0}m(v'),	
\end{equation}
for the symmetric $q_t$'s and the initial condition over a discrete space with respect to the measure $m: V \to [0,\infty)$. 

\begin{prop}[{\em \cite{Delmotte}, \S 3.1, Proposition 3.1}] \label{prop:Delmotte}
	Assume there exists $d > 0$ such that the weighted graph $G = (V,E,w)$ satisfies $VG(d)$, $PI$ and $\Delta$. Then for all $v,v' \in V$ and $t \geq 0$ there exists a constant $C_0 > 0$ independent of $v,v'$ and $t$ such that
	\begin{equation} \label{DelmotteUpper}
		p_t(v,v') \leq C_0m(v')t^{-\frac{d}{2}}. 
	\end{equation}   	
\end{prop}

Delmotte proves a much stronger version of Proposition $\ref{prop:Delmotte}$ under more general volume growth conditions that applies to a diverse range of graphs, but for our purposes we work with Proposition $\ref{prop:Delmotte}$ as it is stated here. Delmotte also goes further to prove that the assumptions of Proposition $\ref{prop:Delmotte}$ are equivalent to a Parabolic Harnack Inequality, which we do not explicitly state here because it will not be necessary to our result. What is important to note though is that Theorem $2.32$ of $\cite{GyryaSaloffCoste}$ dictates that any graph (or more generally metric space) satisfying the Parabolic Harnack Inequality further satisfies the estimate
\begin{equation} \label{NeighbourDecay}
	|p_t(v_1,v_3) - p_t(v_2,v_3)| \leq Cm(v_3) \bigg(\frac{d(v_1,v_2)}{\sqrt{t}}\bigg)^\eta p_{2t}(v_1,v_3) 
\end{equation}
for all $v_1,v_2,v_3 \in V$ and some independent $C, \eta > 0$. Thus, we may assume that when the conditions of Proposition $\ref{prop:Delmotte}$ are satisfied, then so must be $(\ref{NeighbourDecay})$.

\begin{cor} \label{cor:InfDecay} 
	Let $G = (V,E,w)$ be a weighted graph satisfying the assumptions of Proposition $\ref{prop:Delmotte}$. If there exists an $M > 0$ such that $m(v) \leq M$ for all $v \in V$, then for any $x_0 \in \ell^1(V)$ there exists a constant $C > 0$ such that
	\begin{equation} \label{Contraction1}
		\|P_t\|_{1 \to \infty} \leq Ct^{-\frac{d}{2}}.
	\end{equation}
\end{cor}

\begin{proof}
	Since the conditions of Proposition $\ref{prop:Delmotte}$ are satisfied for some $d > 0$, there exists a $C_0 > 0$ such that $(\ref{DelmotteUpper})$ holds. Then apply H\"older's Inequality to the general solution $(\ref{HeatSoln})$ to find that 
	\begin{equation} \label{UltraContractive1}
		|[P_tx_0]_v| \leq C_0Mt^{-\frac{d}{2}} \|x_0\|_1.		
	\end{equation} 
	Taking the supremum over all $v \in V$ gives the desired result. 
\end{proof}

Ultracontractive properties such as that stated in Corollary $\ref{cor:InfDecay}$ of one parameter semigroups have been intensely studied, notably in the seminal work of Varopoulos who provided similar results to $(\ref{Contraction1})$ in a much more general setting, but nonetheless arrived at a similar conclusion $\cite{Varopoulus}$. 

It has been demonstrated (see, for example, page $219$ of $\cite{KellerLenz}$) that if $0 \leq x_{v,0} \leq 1$ for all $v\in V$, then $0 \leq [P_tx_0]_v \leq 1$ for all $v \in V$. The lower bound follows immediately from the positivity of the probabilities $p_t(\cdot,\cdot)$, whereas the upper bound follows from a direct application of H\"older's Inequality and the identity $(\ref{ProbSum})$. Following the comments at the beginning of Section $1.2$ of $\cite{BernicotCoulhonFrey}$, this implies that there exists a $C_{op} > 0$ for which 
\begin{equation}
	\|P_t\|_{p \to p} \leq C_{op}
\end{equation}  
for all $1 \leq p \leq \infty$. These uniform bounds and the ultracontractivity property $(\ref{UltraContractive1})$ can be extended further by the following lemma.

\begin{lem} \label{lem:Ultracontractive} 
	Assume there exists constants $C,C_{op}>0$ such that $\|P_t\|_{1\to \infty} \leq Ct^{-\frac{d}{2}}$ and $\|P_t\|_{1 \to 1} \leq C_{op}$. Then for all $1 \leq p \leq \infty$ there exists a constant $C > 0$, independent of $p$, such that
	\begin{equation} \label{UltraContractive2}
		\|P_t\|_{1 \to p} \leq C t^{-\frac{d}{2}(1 - \frac{1}{p})}.
	\end{equation}  
\end{lem}

\begin{proof}
	We begin by recalling the log-convexity property of $\ell^p$ norms. For any $1 \leq p_0 \leq p_1 \leq \infty$ and $0 < \gamma < 1$ we define $p_\gamma$ by the equation 
	\begin{equation}
		\frac{1}{p_\gamma} = \frac{1 - \gamma}{p_0} + \frac{\gamma}{p_1}.
	\end{equation}
	Then for all $x \in \ell^{p_0}(V)$ we have
	\begin{equation}
		\|x\|_{p_\gamma} \leq \|x\|_{p_0}^{1 - \gamma}\|x\|_{p_1}^\gamma.
	\end{equation} 	
	
	To apply this log-convexity property to our present situation we take $p_0 = 1$ and $p_1 = \infty$. Then $p_\gamma = \frac{1}{1-\gamma}$ and for any $x \in \ell^1(V)$ we have
	\begin{equation}
		\begin{split}
			\|P_tx\|_{p_\gamma} &\leq \|P_tx\|_1^{\frac{1}{p_\gamma}}\|P_tx\|_{\infty}^{1 - \frac{1}{p_\gamma}} \\
			&\leq \|P_t\|_{1 \to 1}^\frac{1}{p_\gamma}\|x\|_1^{\frac{1}{p_\gamma}} \|P_t\|_{1\to \infty}^{1 - \frac{1}{p_\gamma}}\|x\|_1^{1 - \frac{1}{p_\gamma}} \\
			&\leq C_{op}^\frac{1}{p_\gamma}C^{1 - \frac{1}{p_\gamma}}t^{-\frac{d}{2}(1 - \frac{1}{p_\gamma})}\|x\|_1.
		\end{split}
	\end{equation} 
	Thus, taking $\|x\|_1 = 1$ shows $\|P_t\|_{1\to p_\gamma} \leq C_{op}^\frac{1}{p_\gamma}C^{1 - \frac{1}{p_\gamma}}t^{-\frac{d}{2}(1 - \frac{1}{p_\gamma})}$. By varying $\gamma \in (0,1)$ we obtain the result for $1 < p < \infty$ and the endpoints $p = 1,\infty$ are taken care of by assumption. Finally, the bound $C_{op}^\frac{1}{p_\gamma}C^{1 - \frac{1}{p_\gamma}}$ as a function of $p$ is uniformly bounded on $p\in[1,\infty]$, and therefore taking any $C > 0$ to be the supremum of this function will give that
	\begin{equation}
		\|P_tx\|_{1\to p} \leq Ct^{-\frac{d}{2}(1 - \frac{1}{p_\gamma})}\|x\|_1,	
	\end{equation}    
	completing the proof.
\end{proof}

We will see in the coming chapter that our investigations will greatly utilize the case $p = 2$ from Lemma $\ref{lem:Ultracontractive}$, which gives
\begin{equation} \label{UltraContractive3}
	\|P_t\|_{1 \to 2} \leq C_2 t^{-\frac{d}{4}}.	
\end{equation} 
Finally, to avoid the singularity at $t = 0$ in $(\ref{UltraContractive1})$ and $(\ref{UltraContractive3})$, we will use the alternative upper bounds:
\begin{subequations}
	\begin{equation}
		\|P_t\|_{1 \to \infty} \leq \tilde{C}(1 + t)^{-\frac{d}{2}},	
	\end{equation}	
	\begin{equation}
		\|P_t\|_{1 \to 2} \leq \tilde{C}(1 + t)^{-\frac{d}{4}},
	\end{equation}	
\end{subequations}
with a new constant $\tilde{C} > 0$. Note that such an alternative upper bound is possible since for large $t$ these new upper bounds decay at the same rate as the bounds in $(\ref{UltraContractive1})$ and $(\ref{UltraContractive3})$ and for small $t \geq 0$ the operator $P_t$ is well-behaved and finite.  We will also use such an alternative upper bound of $(1+t)^{-\frac{\eta}{2}}$ in $(\ref{NeighbourDecay})$ for the same reason.

This concludes our very brief exploration of the rich and diverse area of random walks on graphs. The results stated in this section will be applied to the phase system $(\ref{PerturbationSystem})$ in the coming sections in order to understand the stability of phase-locked solutions.

\section{A General Stability Theorem} \label{sec:Stability} 

Having now laid a foundation in the theory of random walks on weighted graphs, we are now in a position to return to our discussion of the stability of phase-locked solutions to the phase system $(\ref{PhaseSystem})$. For simplicity, in this section and the next we will often write the perturbation system $(\ref{PerturbationSystem})$ abstractly as an ordinary differential equation in the variable $\psi = \{\psi_v(t)\}_{v\in V}$ as 
\begin{equation} \label{CoupledNetwork2}
	\dot{\psi} = \mathcal{F}(\psi),
\end{equation}  
where $\mathcal{F}: \mathbb{R}^V \to \mathbb{R}^V$ represents the right-hand side of $(\ref{PerturbationSystem})$. Throughout this section we will lay out the Hypotheses necessary on $(\ref{PhaseSystem})$ to eventually provide the main stability theorem for phase-locked solutions to this system. The first of these hypotheses is as follows:

\begin{hyp} \label{Hyp:Neighbours} 
	There exists a finite $D \geq 1$ such that $1 \leq \#N(v) \leq D$ for every $v \in V$. 
\end{hyp}  

Hypothesis $\ref{Hyp:Neighbours}$ says that each element $\psi_v$ is influenced by a finite number of other elements, and that the number of influences on any single element is uniformly bounded. Moreover, we recall from our definition of the model that if $\psi_{v'}$ influences $\psi_v$, then $\psi_v$ influences $\psi_{v'}$, and hence the influence topology is symmetric. 

\begin{hyp} \label{Hyp:Linearization} 
	The coupling function $H:S^1 \to S^1$ satisfies
	\begin{equation} \label{PositiveWeights}
		H'(\bar{\theta}_{v'} - \bar{\theta}_v) = H'(\bar{\theta}_v - \bar{\theta}_{v'}) \geq 0
	\end{equation} 
	for all $v \in V$ and $v' \in N(v)$.
\end{hyp}

What we obtain from Hypothesis $\ref{Hyp:Linearization}$ is our phase-locked solution naturally can be related to an infinite weighted graph, simply denoted $G = (V,E,w)$, with vertex set $V$ and edge set, $E$, contained in the set of all possible influences. Indeed, the symmetric weight function is given by 
\begin{equation}
	w(v,v') = \left\{
     		\begin{array}{cl}
       			H'(\bar{\theta}_{v'} - \bar{\theta}_{v}) & : v' \in N(v)\\
       			0 & : v' \notin N(v). \\
     		\end{array}
   	\right.
\end{equation}   
The condition $(\ref{PositiveWeights})$ guarantees that $w(v,v') = w(v',v) \geq 0$ for all $v,v' \in V$. Furthermore, Hypothesis $\ref{Hyp:Neighbours}$ guarantees that the degree of each vertex is finite and uniformly bounded. 
 
Notice that a necessary condition for there to be an edge between vertices $v,v' \in V$ is that $v' \in N(v)$ (or equivalently $v \in N(v')$). This condition is not sufficient since it could be the case that for some $v \in V$ and $v' \in N(v)$ we have $H'(\bar{\theta}_{v'} - \bar{\theta}_{v}) = 0$, and therefore there is no edge between $v$ and $v'$ by definition of a weight function on a graph. This implies that even if $v' \in N(v)$, the distance between these vertices on the graph $G$ is not guaranteed to be $1$ since there may not be an edge connecting these vertices. This necessitates the following hypothesis.

\begin{hyp} \label{Hyp:GraphDistance} 
	$G$ is a connected graph. Furthermore the associated graph metric, $\rho: V \times V \to [0,\infty)$, satisfies
	\begin{equation} \label{MetricBound}
		\sup_{v \in V, v'\in N(v)}\rho(v,v') < \infty.	
	\end{equation}	
\end{hyp}

Before we are able to apply the results of random walks on infinite graphs, we must point out that the linearization $L_{\bar{\theta}}$ given in $(\ref{Linearization})$ is not in the form that was investigated through random walks. That is, we are missing the $1/m(v)$ term from $(\ref{NormalizedLap})$. If $m(v)$ is positive and independent of $v$ we may simply rescale $t \to m(v) t$, which will apply the appropriate $1/m(v)$ term to obtain the appropriate Laplacian operator. Then the operator $[1/m(v)]L_{\bar{\theta}}$ is in the appropriate form to apply the theory from random walks on graphs.

We will now describe how to overcome this problem when $m(v)$ is not independent of $v$. To begin, note that Hypotheses $\ref{Hyp:Neighbours}$ and $\ref{Hyp:Linearization}$ together give that there exists an $M > 0$ such that
\begin{equation}
	m(v) = \sum_{v' \in N(v)} H'(\bar{\theta}_{v'} - \bar{\theta}_v) \leq M
\end{equation}  
for all $v \in V$. Letting $t \to (M + 1)t$ scales $(\ref{CoupledNetwork2})$ to the equivalent differential equation 
\begin{equation}
	\dot{\psi} = \frac{1}{M + 1}\mathcal{F}(\psi).
\end{equation}
Furthermore, linearizing about the steady-state $\psi = 0$ now results in the linearization 
\begin{equation}
	\tilde{L}_{\bar{\theta}} := \frac{1}{M + 1}L_{\bar{\theta}}.
\end{equation}
We have now normalized the operator $L_{\bar{\theta}}$, and wish to consider a new graph, $\tilde{G}$, so that the measure of each vertex is given by $\tilde{m}(v) = M + 1$ for all $v \in V$. In doing so we will have that $\tilde{L}_{\bar{\theta}}$ is of the proper form to apply the results of the previous section. First, notice that 
\begin{equation}
	\sum_{v' \in V} \frac{w(v,v')}{M + 1} = \sum_{v' \in N(v)} \frac{H'(\bar{\theta}_{v'} - \bar{\theta}_v)}{M+1} \leq \frac{M}{M + 1} < 1. 
\end{equation}
Let us extend the graph $G$ to $\tilde{G}$ by adding a loop at every vertex (an edge which originates and terminates at the same vertex) and augment to a new weight function $\tilde{w}: V \times V \to [0,\infty)$ given by
\begin{equation} \label{LoopConstruction}
	\tilde{w}(v,v') = \left\{
     		\begin{array}{cl}
       			H'(\bar{x}_{v'} - \bar{x}_{v}) & : v' \in N(v),\ v' \neq v\\
       			1 + M - \sum_{v''}H'(\bar{x}_{v''} - \bar{x}_{v}) & : v' = v \\ 
			0 & : v' \notin N(v). \\
     		\end{array}
   	\right.	
\end{equation}
That is, the missing weight for the measure $\tilde{m}$ to be identically $M + 1$ for each $v\in V$ is made up for by the new loop connecting each vertex to itself. Notice that adding loops to a graph does not change the form of the graph Laplacian. Indeed, for each $v \in V$ we have
\begin{equation}
	\begin{split}
		\sum_{v' \in N(v)} w(v,v')(x_{v'} - x_v) &= \sum_{v' \in N(v)} \underbrace{\tilde{w}(v,v')}_\text{$=w(v,v')$}(x_{v'} - x_v) \\
		&= \sum_{v' \in N(v)} \tilde{w}(v,v')(x_{v'} - x_v) + \underbrace{\tilde{w}(v,v)(x_v - x_v)}_\text{$= 0$}  \\
		&= \sum_{v' \in N(v)\cup\{v\}} \tilde{w}(v,v')(x_{v'} - x_v).	
	\end{split}
\end{equation}
The underlying graph will be denoted $\tilde{G} = (V,\tilde{E},\tilde{w})$. Notice that if $G$ is connected then $\tilde{G}$ is also connected since we have not eliminated any edges from $G$ to form $\tilde{G}$. Furthermore, we again have a uniform bound on the weight function $\tilde{w}$ given by $M + 1$. One should also note that Hypothesis $\ref{Hyp:Neighbours}$ also gives that $\tilde{L}_{\bar{\theta}}$ again is a bounded operator on the sequence spaces from Lemma $\ref{lem:NormLaplacian}$.  

\begin{hyp} \label{Hyp:DecayRates} 
	Assume that one of the following is true: 
	\begin{itemize}
		\item If $m(v)$ is independent of $v\in V$, assume there exists a $d\geq 2$ such that the graph $G = (V,E,w)$ satisfies $VG(d)$, $PI$ and $\Delta$. 
		\item If $m(v)$ is not independent of $v \in V$, assume there exists a $d \geq 2$ such that the graph $\tilde{G} = (V,\tilde{E},\tilde{w})$ (as constructed above) satisfies $VG(d)$, $PI$ and $\Delta$. 	
	\end{itemize}	
\end{hyp}

Notice that the assumptions of this hypothesis imply that the graph satisfies the assumptions of Proposition $\ref{prop:Delmotte}$, and therefore we obtain the algebraic decay rates on the transition probabilities of a random walk on the vertices of the graph $\cite{Delmotte}$. This hypothesis then in turn allows one to infer the results of Corollary $\ref{cor:InfDecay}$ and Lemma $\ref{lem:Ultracontractive}$. This leads to the following stability theorem whose proof is left to Section $\ref{sec:StabilityProof}$.

\begin{thm} \label{thm:Stability} 
	Consider the system $(\ref{PerturbationSystem})$ for a twice-differentiable function $H:S^1 \to S^1$ satisfying the Hypothesis $\ref{Hyp:Neighbours}$ and $\ref{Hyp:Linearization}$. Assume further that linearizing about the steady-state $\psi = 0$ leads to a linear operator, $L_{\bar{\theta}}$, satisfying Hypotheses $\ref{Hyp:GraphDistance}$ and $\ref{Hyp:DecayRates}$. Then, there exists an $\epsilon > 0$ for which every $\psi_0 = \{\psi_{v,0}\}_{v \in V}$ with the property that
	\begin{equation}
		\|\psi_0\|_1 \leq \varepsilon,	
	\end{equation}   
	defines a unique solution of $(\ref{PerturbationSystem})$, $\psi(t)$ for all $t \geq 0$, satisfying the following properties:
	\begin{enumerate}
		\item $\psi(0) = \psi_0$.
		\item $\psi(t) \in \ell^p(V)$ for all $1 \leq p \leq \infty$.
		\item There exists a $C > 0$ such that 
		\begin{subequations}
			\begin{equation}
					\|\psi(t)\|_1 \leq C \|\psi_0\|_1, 
				\end{equation}
			\begin{equation}
					\|\psi(t)\|_2 \leq C (1 + t)^{-\frac{d}{4}}\|\psi_0\|_1, 
				\end{equation}
				\begin{equation}
					\|\psi(t)\|_\infty \leq C (1 + t)^{-\frac{d}{2}}\|\psi_0\|_1,
			\end{equation}
		\end{subequations}
			for all $t \geq 0$ and $d\geq 2$ given by Hypothesis~\ref{Hyp:DecayRates}. 
	\end{enumerate}  	
\end{thm} 

We present the following extension of Theorem~\ref{thm:Stability} for completion. It should be noted that the proof is identical to that of Lemma $\ref{lem:Ultracontractive}$ where one simply applies the log-convexity of the $\ell^p$ norms and is therefore omitted. 

\begin{cor}
	Under the assumptions of Theorem~\ref{thm:Stability} there exists a $C > 0$ such that the solution $\psi(t)$ further satisfies
	\begin{equation}
		\|\psi(t)\|_p \leq C(1+t)^{-\frac{d}{2}(1-\frac{1}{p})}\|\psi_0\|_1,
	\end{equation} 
	for all $t\geq 0$ and $p \in [1,\infty]$.
\end{cor}

\begin{rmk} \label{rmk:Digraph}
	It is important to note that the most restrictive assumption is the symmetry condition $(\ref{PositiveWeights})$ of Hypothesis $\ref{Hyp:Linearization}$. When $(\ref{PositiveWeights})$ is broken for even a single index, the graph becomes a directed graph (or digraph) and therefore all of the theory from Section $\ref{sec:Graphs}$ can no longer be applied. This situation would require the development of comparable techniques to obtain similar decay rates for random walks on digraphs which do not appear to be available at this time. 
\end{rmk}

\begin{rmk}
	Although in this work we have assumed that coupling between oscillators is identically given through the function $H$, Theorem~\ref{thm:Stability} could be extended to non-identical coupling functions as well. We have refrained from doing this here for the ease of conveying the results and also due to the fact that it seems impractical to assume that a symmetry condition equivalent to $(\ref{PositiveWeights})$ could hold. That is, for a system with non-identical coupling such as
	\begin{equation}
		\dot{\theta} = \omega_v + \sum_{v' \in N(v)} H(\theta_{v'} - \theta_v,v,v'),
	\end{equation}  
	condition $(\ref{PositiveWeights})$ would need to be replaced with a condition such as
	\begin{equation} \label{NonIdenticalSymmetry}
		H'(\bar{\theta}_{v'} - \bar{\theta}_v,v,v') = H'(\bar{\theta}_{v} - \bar{\theta}_{v'},v',v) \geq 0 
	\end{equation}
	for all $v \in V$ and $v' \in N(v)$. It appears that stability results for non-identical coupling functions would be more realistic if the symmetry condition $(\ref{NonIdenticalSymmetry})$ could merely be replaced with the simpler condition $H'(\bar{\theta}_{v'} - \bar{\theta}_v,v,v') \geq 0$, thus resulting in a digraph. This is exactly what has already been done for finite systems of coupled oscillators in $\cite{ErmentroutStability}$. 
\end{rmk}

\section{Proof of Theorem~\ref{thm:Stability}} \label{sec:StabilityProof} 

Throughout this section we will assume that Hypothesis $\ref{Hyp:Neighbours},\ref{Hyp:Linearization},\ref{Hyp:GraphDistance}$ and $\ref{Hyp:DecayRates}$ hold. We will work through this proof under the assumption that the second case of Hypothesis $\ref{Hyp:DecayRates}$ holds, although the proof using the first case is nearly identical. Following the discussion prior to stating Hypothesis $\ref{Hyp:DecayRates}$, we will apply the appropriate re-parametrization of $t$. Let us further write $(\ref{CoupledNetwork2})$ in the equivalent form
\begin{equation} \label{CoupledNetwork3}
	\dot{\psi} = \tilde{L}_{\bar{\theta}}\psi + \mathcal{G}(\psi),
\end{equation}
where $\mathcal{G}(\psi) = \frac{1}{M+1}\mathcal{F}(\psi) - \tilde{L}_{\bar{\theta}}\psi$ and we suppress the dependence of $\psi$ on $t$. Notice that $\mathcal{G}(0) = 0$ and its the derivative satisifies $\mathcal{G}'(0) = 0$. Denoting $P_t = e^{\tilde{L}_{\bar{\theta}}t}$ to be the semigroup generated by the linearization $\tilde{L}_{\bar{\theta}}$ for all $t \geq 0$, we arrive at the equivalent formulation of $(\ref{CoupledNetwork3})$ given by
\begin{equation} \label{IntegralForm}
	\psi(t) = P_t \psi(0) + \int_0^t P_{t-s}\mathcal{G}(\psi(s))ds	
\end{equation} 
for any $t \geq 0$. Then any function $\psi(t)$ which satisfies $(\ref{IntegralForm})$ satisfies the differential equation $(\ref{PerturbationSystem})$ for $t \geq 0$.

Let $Q: \mathbb{R}^V \to \mathbb{R}$ be the operator acting upon the elements $x = \{x_v\}_{v\in V}$ by 
\begin{equation} \label{Quadratic}
	Q(x) = \sum_{v \in V} \sum_{v' \in N(v)} |x_{v'} - x_v|^2.
\end{equation}
Clearly $Q(x) \geq 0$ for all $x$, and furthermore using the Parallelogram Law one can see that for any $x \in \ell^2(V)$ we have
\begin{equation} \label{QuadBound}
	0 \leq Q(x) \leq 2D\|x\|_2^2, 
\end{equation} 
where we recall that from Hypothesis $\ref{Hyp:Neighbours}$ we have that $\#N(v) \leq D$ for all $v \in V$. Furthermore, $\sqrt{Q(\cdot)}$ defines a seminorm, and therefore satisfies the triangle inequality (for example, see Lemma $4.3$ of $\cite{Jorgensen}$). This leads to the first result. 

\begin{lem} \label{lem:Boundedness1} 
	For any $\psi \in \ell^2(V)$, there exists a $K > 0$, depending on $\|\psi\|_2$, such that
	\begin{equation}
		\|\mathcal{G}(\psi)\|_1 \leq KQ(\psi).
	\end{equation}
\end{lem}

\begin{proof}
	Let us write $\delta := \|\psi\|_2$. We then have that $\|\psi\|_\infty \leq \delta$ and hence $|\psi_{v'} - \psi_v| \leq 2\delta$ for all $v,v' \in V$. Since $H \in C^2(\mathbb{R})$ we can define
	\begin{equation}
		K_1(\delta) := \sup_{|x| \leq 2\delta} |H''(x)| < \infty.
	\end{equation} 	
	By Taylor's Theorem, for all $v \in V$ and $v' \in N(v)$ we have
	\begin{equation} 
		|H(\bar{\theta}_{v'} + \psi_{v'} - \bar{\theta}_v -\psi_v) - H(\bar{\theta}_{v'} - \bar{\theta}_v) - H'(\bar{\theta}_{v'} - \bar{\theta}_v)(\psi_{v'} - \psi_v)| \leq \frac{K_1(\delta)}{2}|\psi_{v'} - \psi_v|^2.
	\end{equation} 
	Then recalling $\mathcal{G}(0) = 0$ and $\mathcal{G}'(0) = 0$, and using the previous inequality we get
	\begin{equation}
		\begin{split}
			\|\mathcal{G}(\psi)\|_1 &= \|\mathcal{G}(\psi) - \mathcal{G}(0) - \mathcal{G}'(0)\psi\|_1 \\
			&= \frac{1}{M+1}\sum_v \bigg|\sum_{v' \in N(v)} H(\bar{\theta}_{v'} + \psi_{v'} - \bar{\theta}_v -\psi_v) - H(\bar{\theta}_{v'} - \bar{\theta}_v) \\
			 &\ \ \ \ \ - H'(\bar{\theta}_{v'} - \bar{\theta}_v)(\psi_{v'} - \psi_v)\bigg| \\
			&\leq \frac{K_1(\delta)}{2(M+1)} \sum_v \sum_{v' \in N(v)} |\psi_{v'} - \psi_v|^2 \\
			&= \frac{K_1(\delta)}{2(M+1)} Q(\psi),  
		\end{split}
	\end{equation}
	completing the proof of the lemma. 
\end{proof} 

Now from the decay rates in Section $\ref{subsec:RandomWalks}$, for all $t \geq 0$ we have the following decay estimates for the semigroup $P_t$:
\begin{subequations} \label{DecayEstimates} 
	\begin{equation} \label{PtContraction}
		\|P_t\|_{p \to p} \leq C_{op},\ \ {\rm for\ all}\ \ 1 \leq p \leq \infty, 
	\end{equation}
	\begin{equation} \label{2Decay}
		\|P_t\|_{1 \to 2} \leq C_1(1 + t)^{-\frac{d}{4}},
	\end{equation}
	\begin{equation} \label{InftyDecay}
		\|P_t\|_{1 \to \infty} \leq C_1(1 + t)^{-\frac{d}{2}},
	\end{equation}
\end{subequations}
for some $C_{op}, C_1 > 0$. There is also one more important estimate which must be established in the following lemma.  

\begin{lem} \label{lem:QForm} 
	Let $\eta > 0$ be the associated value to $P_t$ that satisfies the estimate $(\ref{NeighbourDecay})$. For all $\psi \in \ell^2(V)$, there exists a constant $C_Q > 0$ independent of $\psi$ such that 
	\begin{equation} 
		\sqrt{Q(P_t\psi)} \leq C_Q (1 + t)^{-\frac{\eta}{2}} \|P_{t}|\psi| \|_2,
	\end{equation}  
	where $|\psi| = \{|\psi_v|\}_{v\in V}$.
\end{lem}

\begin{proof}
	To begin, since Hypothesis $\ref{Hyp:DecayRates}$ guarantees that the measure of each vertex is uniformly bounded, we combine this statement with $(\ref{NeighbourDecay})$ and the uniform boundedness of the metric given in Hypothesis $\ref{Hyp:GraphDistance}$ to find that there exists a $C > 0$ (independent of $v,v'$ and $t$) such that 
	\begin{equation}
		|p_t(v,v'') - p_t(v',v'')| \leq C (1 + t)^{-\frac{\eta}{2}} p_{2t}(v,v''),
	\end{equation} 
	for all $v,v'' \in V$, $v' \in N(v)$. Then for any $\psi \in \ell^2(V)$ we have 
	\begin{equation}
		\begin{split}
			|[P_t\psi]_v - [P_t\psi]_{v'}| &\leq \sum_{v'' \in V} |p_t(v,v'') - p_t(v',v'')| |\psi_{v''}| \\
			&\leq C(1 + t)^{-\frac{\eta}{2}} \sum_{v'' \in V} p_{2t}(v,v'') |\psi_{v''}| \\
			&= C(1 + t)^{-\frac{\eta}{2}}[P_{2t}|\psi|]_v. 
		\end{split}
	\end{equation}
	This in turn gives
	\begin{equation}
		\begin{split}
			\sqrt{Q(P_t\psi)} &= \sqrt{\sum_{v \in V} \sum_{v' \in N(v)} |[P_t \psi]_{v'} - [P_t \psi]_v|^2} \\
			&\leq C(1 + t)^{-\frac{\eta}{2}}\sqrt{\sum_{v \in V} \sum_{v' \in N(v)} |[P_{2t}|\psi|]_v|^2} \\
			&\leq CD(1 + t)^{-\frac{\eta}{2}}\sqrt{\sum_{v \in V} |[P_{2t}|\psi|]_v|^2} \\
			&\leq CD(1 + t)^{-\frac{\eta}{2}}\|P_{2t}|\psi|\|_2.  
		\end{split}
	\end{equation}
	 Finally, using the fact that $P_{2t} = P_tP_t$ and the decay estimate $\|P_t\|_{2\to 2} \leq C_{op}$ from $(\ref{PtContraction})$ we arrive at the final result
	 \begin{equation}
	 	\sqrt{Q(P_t\psi)} \leq CC_{op}D(1 + t)^{-\frac{\eta}{2}}\|P_{t}|\psi|\|_2.	
	 \end{equation}
\end{proof} 

Let us now consider an initial condition $\psi_0 \in \ell^1(V)$. We want to prove that if $\|\psi_0\|_1$ is chosen small enough, there exists a solution $\psi(t)$ to $(\ref{PerturbationSystem})$ with $\psi(0) = \psi_0$ belonging to the space
\begin{equation} \label{USpace} 
	\begin{split}
	X = \bigg\{\psi(t)\bigg|\ &\psi(0) = \psi_0,\ \|\psi(t)\|_2 \leq 2C_1(1 + t)^{-\frac{d}{4}}\|\psi_0\|_1\ \ {\rm and}\  \\
		&\sqrt{Q(\psi(t))} \leq 2C_1C_Q(1 + t)^{-\frac{d}{4}-\frac{\eta}{2}}\|\psi_0\|_1,\  \forall \ t\geq 0 \bigg\} 
	\end{split}	
\end{equation}  
where $C_1 > 0$ is the constant taken from the decay estimates $(\ref{DecayEstimates})$ and $C_Q > 0$ is the constant from Lemma $\ref{lem:QForm}$. Prior of showing the existence of a solution to $\ref{PerturbationSystem}$, we show that functions belonging to $X$ indeed satisfy the additional statements of Theorem~\ref{thm:Stability}. We require the following lemma, which has been repurposed from $\cite{IntegralLemma}$.

\begin{lem}[{\em \cite{IntegralLemma}, \S 3, Lemma 3.2}]{\bf (Restated)} \label{lem:IntegralLemma} 
	Let $\gamma_1, \gamma_2$ be positive real numbers. If $\gamma_1,\gamma_2 \neq 1$ or if $\gamma_1 = 1 < \gamma_2$ then there exists a $C_{\gamma_1,\gamma_2} > 0$ such that
	\begin{equation}
		\int_0^t (1 + t - s)^{- \gamma_1}(1 + s)^{-\gamma_2}ds \leq C_{\gamma_1,\gamma_2} (1 + t)^{-\min\{\gamma_1 + \gamma_2 - 1, \gamma_1, \gamma_2\}},
	\end{equation}  	 
	for all $t \geq 0$.
\end{lem}

We now provide the necessary results to proving Theorem~\ref{thm:Stability}. We first define the mapping 
\begin{equation} \label{TMapping}
	T\psi(t) = P_t\psi_0 + \int_0^t P_{t-s}\mathcal{G}(\psi(s))ds,
\end{equation}
where $t \geq 0$. Notice that a fixed point of this mapping belonging to $X$ for all $t \geq 0$ will satisfy the differential equation $(\ref{PerturbationSystem})$. We present the following proposition.

\begin{prop} \label{prop:Decays} 
	Let $t_0 > 0$ and assume that $\psi(t)$ satisfies 
	\begin{equation}
		\|\psi(t)\|_2 \leq 2C_1(1 + t)^{-\frac{d}{4}}\|\psi_0\|_1
	\end{equation}
	and
	\begin{equation}
		 \sqrt{Q(\psi(t))} \leq 2C_1C_Q(1 + t)^{-\frac{d}{4} - \frac{\eta}{2}}\|\psi_0\|_1
	\end{equation}
	for all $0 \leq t \leq t_0$. Then there exists an $\varepsilon_0 > 0$, independent of $t_0$, such that if $\|\psi_0\|_1 \leq \varepsilon_0$ we have 
	\begin{subequations} \label{MappingDecay1}
		\begin{align}
			\|T\psi(t)\|_2 &\leq 2C_1(1 + t)^{-\frac{d}{4}}\|\psi_0\|_1, \\
			\sqrt{Q(T\psi(t))} &\leq 2C_1C_Q(1 + t)^{-\frac{d}{4} - \frac{\eta}{2}}\|\psi_0\|_1
		\end{align}
	\end{subequations}
	for all $0 \leq t \leq t_0$. Furthermore, there exists a $C_\infty > 0$, independent of $t_0$, such that
		\begin{subequations} \label{MappingDecay2}
			\begin{align} 
				\|T\psi(t)\|_1 &\leq C_\infty \|\psi_0\|_1, \label{ell1NormBound}, \\
				\|T\psi(t)\|_\infty &\leq C_\infty (1 + t)^{-\frac{d}{2}}\|\psi_0\|_1.
			\end{align}
		\end{subequations} 
	for all $0 \leq t \leq t_0$.
\end{prop}

\begin{proof}
	Begin by considering $\|\psi_0\|_1 \leq \frac{1}{2C_1}$ so that $\|\psi(t)\|_2 \leq 1$ for all $0 \leq t \leq t_0$. This in turn guarantees the existence of a uniform $K > 0$ such that the results of Lemma $\ref{lem:Boundedness1}$ hold. Then using the decay estimate $(\ref{2Decay})$ we now have
	\begin{equation}
		\begin{split}
			\|T\psi(t)\|_2 &\leq \|P_t\psi_0\|_2 + \int_0^t \|P_{t-s}\mathcal{G}(\psi(s))\|_2 ds \\
			&\leq C_1(1+t)^{-\frac{d}{4}}\|\psi_0\|_1 + C_1\int_0^t (1+t-s)^{-\frac{d}{4}}\|\mathcal{G}(\psi(s))\|_2 ds \\
			&\leq C_1(1+t)^{-\frac{d}{4}}\|\psi_0\|_1 + C_1K\int_0^t (1+t-s)^{-\frac{d}{4}}Q(\psi(s)) ds \\
			&\leq C_1(1+t)^{-\frac{d}{4}}\|\psi_0\|_1 + 4C_1^3C_Q^2K \int_0^t(1+t-s)^{-\frac{d}{4}}(1 + s)^{-\frac{d}{2} - \eta}\|\psi_0\|_1^2 ds    
		\end{split}
	\end{equation} 
	for all $0 \leq t \leq t_0$. From Lemma $\ref{lem:IntegralLemma}$, there exists $C_{\frac{d}{4},\frac{d}{2} + \eta} > 0$ such that
	\begin{equation}
		\int_0^t (1 + t - s)^{-\frac{d}{4}}(1 + t)^{-\frac{d}{2} - \eta}ds \leq C_{\frac{d}{4},\frac{d}{2} + \eta}(1 + t)^{-\frac{d}{4}}.	
	\end{equation}
	Thus,
	\begin{equation}
		\|T\psi(t)\|_2 \leq [C_1 + 4C_1^3C_Q^2C_{\frac{d}{4},\frac{d}{2} + \eta}K\|\psi_0\|_1](1 + t)^{-\frac{d}{4}}\|\psi_0\|_1.	
	\end{equation}
	
	Similarly, we use the bounds given in Lemma $\ref{lem:QForm}$ to obtain
	\begin{equation}
  		\begin{split}
   	 		\sqrt{Q(T\psi(t))} &\leq \sqrt{Q(P_t\psi_0)} + \int_0^t \sqrt{Q(P_{t-s}\mathcal{G}(\psi(s)))}ds \\
    			&\leq C_Q(1+t)^{-\frac{\eta}{2}}\|P_{t}\psi_0\|_2 + C_Q \int_0^t (1 + t - s)^{-\frac{\eta}{2}}\|P_{t-s}|\mathcal{G}(\psi(s))\||_2 ds\\
    			&\leq C_1C_Q(1 + t)^{-\frac{d}{4} - \frac{\eta}{2}}\|\psi_0\|_1 \\
    			&\qquad + 4C_1^3C_Q^3K\int_0^t(1+t-s)^{-\frac{d}{4} - \frac{\eta}{2}}(1 + s)^{-\frac{d}{2} - \eta}\|\psi_0\|_1^2 ds.  
  		\end{split}
	\end{equation}  
	From Lemma $\ref{lem:IntegralLemma}$, there exists a $C_{\frac{d}{4} + \frac{\eta}{2}, \frac{d}{2} + \eta} > 0$ such that
	\begin{equation}
		\int_0^t(1+t-s)^{-\frac{d}{4} - \frac{\eta}{2}}(1 + s)^{-\frac{d}{2} - \eta} ds \leq C_{\frac{d}{4} + \frac{\eta}{2}, \frac{d}{2} + \eta}(1 + t)^{-\frac{d}{4} - \frac{\eta}{2}}.,	
	\end{equation}
	which implies that 
	\begin{equation}
		\sqrt{Q(T\psi(t))} \leq [C_1C_Q + 4C_1^3C_Q^3C_{\frac{d}{4} + \frac{\eta}{2}, \frac{d}{2} + \eta}K](1 + t)^{-\frac{d}{4} - \frac{\eta}{2}}\|\psi_0\|_1	
	\end{equation}
	Therefore, taking
		\begin{equation}
			\varepsilon_0 := \min\bigg\{\frac{1}{2C_1},\frac{1}{4C_1^3C_Q^2C_{\frac{d}{4},\frac{d}{2} + \eta}K},\frac{1}{4C_1^3C_Q^2C_{\frac{d}{4},\frac{d}{2} + \eta}K}\bigg\}
		\end{equation}	
	gives the decay estimates (\ref{MappingDecay1}).
	
	We now turn to proving the remaining decay estimates (\ref{MappingDecay2}). First, note that $(\ref{PtContraction})$ details that $\|P_t\|_{1 \to 1} \leq C_{op}$ for all $t \geq 0$. Then, for all $0 \leq t \leq t_0$ we have
	\begin{equation}
		\begin{split}
		\|T\psi(t)\|_1 &\leq \|P_t\psi_0\|_1 + \int_0^t \|P_{t-s}\mathcal{G}(\psi(s))\|_1ds \\
		&\leq C_{op}\|\psi_0\|_1 + C_{op}\int_0^t \|\mathcal{G}(\psi(s))\|_1 ds \\
		&\leq C_{op}\|\psi_0\|_1 + C_{op}K\int_0^t Q(\psi(s))ds \\
		&\leq C_{op}\|\psi_0\|_1 + 4C_{op}C_1^2C^2_QK\int_0^t (1 + s)^{-\frac{d}{2} - \eta}\|\psi_0\|_1^2 ds \\
		&= C_{op}\|\psi_0\|_1 + \frac{4C_{op}C_1^2C^2_QK}{\frac{d}{2} + \eta}[1 - (1 + t)^{1 - \frac{d}{2} - \eta}]\|\psi_0\|_1^2.     
		\end{split}	
	\end{equation} 
	Since $d \geq 2$ and $\eta > 0$ we have that $[1 - (1 + t)^{1 - \frac{d}{2} - \eta}] \leq 1$ for all $t \geq 0$. Thus, $\|T\psi(t)\|_1$ is uniformly bounded by a constant multiple of $\|\psi_0\|_1$, depending only on $\varepsilon_0$, for all $0 \leq t \leq t_0$, proving the first bound of (\ref{MappingDecay2}). 
	
	The $\ell^\infty(V)$ decay follows in a similar manner to above, although we now use the decay estimate $(\ref{InftyDecay})$. In this case we now have
	\begin{equation}
		\begin{split}
			\|T\psi(t)\|_\infty &\leq \|P_t\psi_0\|_\infty + \int_0^t \|P_{t-s}\mathcal{G}(\psi(s))\|_\infty ds \\ 
			&\leq C_1(1 + t)^{-\frac{d}{2}}\|\psi_0\|_1 + C_1\int_0^t (1 + t - s)^{-\frac{d}{2}} \|\mathcal{G}(\psi(s))\|_1 ds \\
			&\leq C_1(1 + t)^{-\frac{d}{2}}\|\psi_0\|_1 + 4C_1^3C^2_QK\int_0^t (1 + t - s)^{-\frac{d}{2}}(1 + s)^{-\frac{d}{2} - \eta}\|\psi_0\|_1^2 ds. 
		\end{split}	
	\end{equation} 
	Applying Lemma $\ref{lem:IntegralLemma}$ with $\gamma_1 = \frac{d}{2} \geq 1$ and $\gamma_2 = \frac{d}{2} + \eta > 1$ gives that there exists $C_{\frac{d}{2},\frac{d}{2} + \eta} > 0$ such that
	\begin{equation}
		\int_0^t (1 + t - s)^{-\frac{d}{2}}(1 + s)^{-\frac{d}{2} - \eta} ds \leq C_{\frac{d}{2},\frac{d}{2} + \eta}(1 + t)^{-\frac{d}{2}}.	
	\end{equation} 
	Therefore, we have the decay estimate
	\begin{equation}
		\|T\psi(t)\|_\infty \leq [C_1 + 4C_1^3C_{\frac{d}{2},\frac{d}{2} + \eta}C^2_QK\|\psi_0\|_1] (1 + t)^{-\frac{d}{2}}\|\psi_0\|_1,	
	\end{equation}
	which holds for all $0 \leq t \leq t_0$, and is independent of $t_0$. This concludes the proof of the proposition. 
\end{proof} 

One can see that Proposition~\ref{prop:Decays} gives that
\begin{equation}
	T:X \to X
\end{equation}
is well-defined when $\|\psi_0\|_1 \leq \varepsilon$. It should be noted that the further estimates given in (\ref{MappingDecay2}) show that upon obtaining a fixed point of $T$ in $X$, we necessarily satisfy all the estimates stated in Theorem~\ref{thm:Stability}. Moreover, recall that $\ell^1(V) \subsetneq \ell^p(V)$ for all $1 < p \leq \infty$, and therefore Proposition~\ref{prop:Decays} dictates that a fixed point of $T$ in $X$ immediately belongs to $\ell^p(V)$ for all $p\in[0,\infty]$. We now show the existence of a fixed point of $T$, which therefore then in turn proves Theorem~\ref{thm:Stability}. 

\begin{proof}[Proof of Theorem~\ref{thm:Stability}] 
	First, since $\mathcal{G}'(0) = 0$, we may consider a $\delta > 0$ such that 
	\begin{equation} \label{LipschitzConstant}
		\sup_{\|x\|_2 \leq \delta} \|\mathcal{G}'(x)\|_{2\to 2} \leq \frac{1}{2C_{op}}.
	\end{equation}
	Notice that such a uniform bound in guaranteed by Hypotheses $\ref{Hyp:Neighbours}$, $\ref{Hyp:Linearization}$ and the smoothness of the function $H$. Then when $\|\psi_0\|_1 \leq \frac{\delta}{2C_1}$, we have that $2C_1(1 + t)^{-\frac{d}{4}}\|\psi_0\|_1 \leq \delta$. Therefore, let 
	\begin{equation} \label{Epsilon}
		\varepsilon := \min\bigg\{\varepsilon_0,\frac{\delta}{2C_1}\bigg\} > 0,
	\end{equation}
	where we use $\varepsilon_0$ to denote the constant required by Proposition~\ref{prop:Decays}. We therefore take $\|\psi_0\|_1 \leq \varepsilon$.
	
	Now consider the space 
	\begin{equation} \label{U1Space} 
		\begin{split}
			X_1 = \bigg\{\psi(t),\ t\in [0,1]\bigg|\ &\psi(0) = \psi_0,\ \|\psi(t)\|_2 \leq 2C_1(1 + t)^{-\frac{d}{4}}\|\psi_0\|_1\ \ {\rm and}\  \\
			&\sqrt{Q(\psi(t))} \leq 2C_1C_Q(1 + t)^{-\frac{d}{4}-\frac{\eta}{2}}\|\psi_0\|_1,\ 0 \leq t \leq 1 \bigg\}. 
		\end{split}	
	\end{equation}
	Then by our choice of $\varepsilon$ in $(\ref{Epsilon})$ we have that
	\begin{equation}
		T: X_1 \to X_1
	\end{equation}  
	is well-defined for $0 \leq t \leq 1$. Let us consider the metric on the space $X_1$ given by
	\begin{equation} \label{U1Metric}
		\rho_{X_1}(\psi_1(t),\psi_2(t)) := \sup_{t \in [0,1]} \|\psi_1(t) - \psi_2(t)\|_2,
	\end{equation}
	for any $\psi_1(t),\psi_2(t) \in X_1$. Using the fact that $\|P_t\|_{2 \to 2} \leq C_{op}$ for all $t \geq 0$, for any $\psi_1(t),\psi_2(t)\in X_1$ this metric gives 
	\begin{equation}
		\begin{split}
			\rho_{X_1}(T\psi_1(t),T\psi_2(t)) &\leq \sup_{t \in [0,1]} \int_0^t \|P_t[\mathcal{G}(\psi_1(s)) - \mathcal{G}(\psi_2(s))]\|_2ds \\
			&\leq	 \sup_{t \in [0,1]} C_{op}\int_0^t \|\mathcal{G}(\psi_1(s)) - \mathcal{G}(\psi_2(s))\|_2ds \\
			&\leq \sup_{t \in [0,1]} \frac{1}{2}\int_0^t \|\psi_1(s) - \psi_2(s)\|_2ds \\
			&\leq \sup_{t \in [0,1]} \frac{t}{2} \rho_{X_1}(\psi_1(t),\psi_2(t)) \\ 
			&\leq \frac{1}{2} \rho_{X_1}(\psi_1(t),\psi_2(t)),
		\end{split}
	\end{equation}
	where we have used the Lipschitz property 
	\[
		\|\mathcal{G}(\psi_1(t)) - \mathcal{G}(\psi_2(t))\|_2 \leq \frac{1}{2C_{op}}\|\psi_1(t) - \psi_2(t)\|_2
	\]	
	which follows from $(\ref{LipschitzConstant})$ and our choice of $\varepsilon > 0$. Thus, $T: X_1 \to X_1$ is a contraction. We provide the following claim, which will be proved after we have completed this proof.
	
	\begin{claim} \label{claim:U1Complete}
		$X_1$ is complete with respect to the metric $(\ref{U1Metric})$.
	\end{claim}
	
	Therefore by the contraction mapping principle there exists a unique fixed point, $\psi_1^*(t) \in X_1$. That is, we have identified a unique solution to the differential equation $(\ref{PerturbationSystem})$ satisfying the decay rates of the space $X$ for $t \in [0,1]$. 
	
	We now proceed by induction. Let us assume that for some positive integer $n \geq 1$ there exists a unique solution to the differential equation $(\ref{PerturbationSystem})$ satisfying the decay rates of the space $X$ for $t \in [0,n]$. That is, there exists a unique fixed point $\psi_n^*(t)$ to $T$ on the interval $[0,n]$ with the properties that $\psi_n^*(0) = \psi_0$, $\|\psi_n^*(t)\|_2 \leq 2C_1(1 + t)^{-\frac{d}{4}}\|\psi_0\|_1$ and $\sqrt{Q(\psi_n^*(t))} \leq 2C_1C_Q(1 + t)^{-\frac{d}{4}-\frac{\eta}{2}}\|\psi_0\|_1$ for all $t \in [0,n]$. We wish to use this function to extend to a solution on $[0,n+1]$. 
	 
	 Begin by defining the space
	 \begin{equation}
	 	\begin{split}
			X_{n+1} = \bigg\{&\psi(t),\ t\in [0,n+1]\bigg|\ \psi(t) = \psi_n^*(t)\ t \in [0,n],\ \|\psi(t)\|_2 \leq 2C_1(1 + t)^{-\frac{d}{4}}\|\psi_0\|_1\\
			& \ {\rm and}\ \sqrt{Q(\psi(t))} \leq 2C_1C_Q(1 + t)^{-\frac{d}{4}-\frac{\eta}{2}}\|\psi_0\|_1, \ n \leq t \leq n+1 \bigg\}. 
		\end{split}		
	 \end{equation}
Again, from out choice of $\varepsilon > 0$, Proposition~\ref{prop:Decays} guarantees that 
\begin{equation}
	T:X_{n+1} \to X_{n+1}
\end{equation}
is well-defined. Let us consider the metric on $X_{n+1}$ given by
\begin{equation} \label{Un+1Metric}
	\rho_{X_{n+1}}(\psi_1(t),\psi_2(t)) := \sup_{t \in [0,n+1]} \|\psi_1(t) - \psi_2(t)\|_2.
\end{equation}
for any $\psi_1(t),\psi_2(t) \in X_{n+1}$. A proof nearly identical to that of the proof of Claim $\ref{claim:U1Complete}$ shows that $X_{n+1}$ is complete with respect to the metric $(\ref{Un+1Metric})$. 

Now, for any $\psi_1(t),\psi_2(t) \in X_{n+1}$ we have that $\psi_1(t) = \psi_2(t)$ for all $t \in [0,n]$. This then gives 
	\begin{equation}
		\begin{split}
			\rho_{X_{n+1}}(T\psi_1(t),T\psi_2(t)) &\leq \sup_{t \in [n,n+1]} \int_n^t \|P_t[\mathcal{G}(\psi_1(s)) - \mathcal{G}(\psi_2(s))]\|_2ds \\
			&\leq \sup_{t \in [n,n+1]} \frac{1}{2}\int_n^t \|\psi_1(s) - \psi_2(s)\|_2ds \\
			&\leq \sup_{t \in [n,n+1]} \frac{t - n}{2} \rho_{X_{n+1}}(\psi_1(t),\psi_2(t)) \\
			&\leq \frac{1}{2} \rho_{X_{n+1}}(\psi_1(t),\psi_2(t)),
		\end{split}
	\end{equation} 
	showing that $T:X_{n+1}\to X_{n+1}$ is a contraction. By the contraction mapping principle, there exists a unique solution $\psi^*_{n+1}(t)$ which extends the solution $\psi_n^*(t)$ onto the interval $[0,n+1]$ and satisfying the required decay rates on this interval. Therefore, there exists a solution to the differential equation $(\ref{PerturbationSystem})$ for arbitrarily large values of $t$ and satisfies the decay rates of the space $X$. Per the discussion following Proposition~\ref{prop:Decays}, this therefore gives the results of Theorem~\ref{thm:Stability}.
\end{proof} 

We conclude with the proof of Claim $\ref{claim:U1Complete}$.

\begin{proof}[Proof of Claim~\ref{claim:U1Complete}] 
	Let $\{\psi_n(t)\}_{n = 1}^\infty$ be a Cauchy sequence in $X_1$. For each fixed $t \in [0,1]$, $\{\psi_n(t)\}_{n=1}^\infty$ forms a Cauchy sequence in the complete space $\ell^2(V)$. Hence, there exists a pointwise limit to the sequence, denoted $\psi(t)$ which belongs to $\ell^2(V)$ for all $t \in [0,1]$. Furthermore, the uniformity of the metric $\rho_{X_1}$ further implies that 
	\begin{equation} \label{X1Convergence}
		\lim_{n \to \infty} \rho_{X_1}(\psi_n(t),\psi(t)) = 0.
	\end{equation}
	It therefore only remains to show that $\psi(t) \in X_1$. 
	
	Let $\varepsilon > 0$ be arbitrary. From $(\ref{X1Convergence})$ there exists $N \geq 0$ such that for all $n \geq N$ and $t \in [0,1]$ we have
	\begin{equation}
		\|\psi_n(t) - \psi(t)\|_2 < \varepsilon.
	\end{equation}
	Then for all $t\in [0,1]$ we have
	\begin{equation}
		\|\psi(t)\|_2 \leq \|\psi_n(t) - \psi(t)\|_2 + \|\psi_n(t)\|_2 < \varepsilon + 2C_1(1 + t)^{-\frac{d}{4}}\|\psi_0\|_1.
	\end{equation}
	Letting $\varepsilon \to 0^+$ gives that $\|\psi(t)\|_2 \leq 2C_1(1 + t)^{-\frac{d}{4}}\|\psi_0\|_1$ for all $t \in [0,1]$.
	
Now for the final condition. Notice that $(\ref{QuadBound})$
dictates that since $\rho_{X_1}(\psi_n(t),\psi(t)) \to 0$
as $n \to \infty$, we have that 
\begin{equation}
  \lim_{n \to \infty} \sup_{t\in [0,1]}\sqrt{Q(\psi_n(t) - \psi(t))} \leq 2D\lim_{n \to \infty}\rho_{X_1}(\psi_n(t),\psi(t))  = 0.
\end{equation} 
	Thus, we merely repeat the previous arguments showing the bounds on $\|\psi(t)\|_2$ to obtain the appropriate bound on $\sqrt{Q(\psi(t))}$. This completes the proof of the claim. 
\end{proof} 

\section{Applications} \label{sec:Applications} 

In this section we will apply the results of Theorem~\ref{thm:Stability} to a phase system which was analyzed in $\cite{MyWork}$. Here we will take $V = \mathbb{Z}^2$ and consider the system of coupled oscillators given by
\begin{equation} \label{MyPhaseSystem1}
	\dot{\theta}_{i,j} = \omega + \sin(\theta_{i+1,j} - \theta_{i,j}) + \sin(\theta_{i-1,j} - \theta_{i,j}) + \sin(\theta_{i,j+1} - \theta_{i,j}) + \sin(\theta_{i,j-1} - \theta_{i,j}),
\end{equation}
for all $(i,j)\in\mathbb{Z}^2$ and a fixed $\omega \in \mathbb{R}$. That is, for each $(i,j) \in \mathbb{Z}^2$, we have that $N((i,j)) = \{(i\pm1,j),(i,j\pm1)\}$. For the ease of notation, we will follow $\cite{MyWork}$ to compactly write $(\ref{MyPhaseSystem1})$ as
\begin{equation} \label{MyPhaseSystem2}
	\dot{\theta}_{i,j} = \omega + \sum_{i',j'} \sin(\theta_{i',j'} - \theta_{i,j}),	
\end{equation}
for all $(i,j) \in \mathbb{Z}^2$. 

Our interest lies in phase-locked solutions to $(\ref{MyPhaseSystem2})$ of the form $(\ref{PhaseAnsatz})$ with $\Omega = \omega$. Solving for the phase-lags, $\bar{\theta} = \{\bar{\theta}_{i,j}\}_{(i,j)\in\mathbb{Z}^2}$ requires one to solve 
\begin{equation} \label{MyPhaseLags}
	\sum_{i',j'} \sin(\bar{\theta}_{i',j'} - \bar{\theta}_{i,j}) = 0	
\end{equation}
for all $(i,j) \in \mathbb{Z}^2$. Having obtained a solution to $(\ref{MyPhaseLags})$, one may follow the general framework of Section $\ref{sec:Perturbation}$ by applying a slight perturbation $\psi = \{\psi_{i,j}\}_{(i,j)\in\mathbb{Z}^2}$ to the phase-locked solution to arrive at the perturbation system
\begin{equation} \label{PerturbationSystem2}
	\dot{\psi}_{i,j} = \sum_{i',j'}  \sin(\bar{\theta}_{i',j'} + \psi_{i',j'} - \bar{\theta}_{i,j} - \psi_{i,j}) 
\end{equation} 
for all $(i,j)\in\mathbb{Z}^2$. The reader is reminded that $\psi = 0$ is a steady-state solution to this perturbation system and its stability corresponds to the stability of the phase-locked solution $\{\omega t + \bar{\theta}_{i,j}\}_{(i,j)\in\mathbb{Z}^2}$. 

Throughout the following subsections we will provide examples of phase-lags that solve $(\ref{MyPhaseLags})$ and thus provide phase-locked solutions to $(\ref{MyPhaseSystem2})$. Furthermore, we will demonstrate how one may apply Theorem~\ref{thm:Stability} to find that these phase-locked solutions are stable. It should immediately be noted that Hypothesis $\ref{Hyp:Neighbours}$ is satisfied by our phase system $(\ref{MyPhaseSystem2})$ since each oscillator is influenced by exactly the four oscillators corresponding to the four nearest-neighbours on the lattice $\mathbb{Z}^2$. Furthermore, since the coupling function $H(x) = \sin(x)$ is an odd function, $H'(x) = \cos(x)$ is an even function and therefore the symmetry requirement $(\ref{PositiveWeights})$ will be met, although the positivity requirement will need to be checked on a case-by-case basis.

\subsection{Stability of the Trivial Solution} \label{subsec:Trivial} 

We begin by illustrating an application of Theorem~\ref{thm:Stability} with the simplest solution to $(\ref{MyPhaseLags})$, the trivial solution. Since $\sin(0) = 0$, taking $\bar{\theta}_{i,j} = 0$ for all $(i,j) \in \mathbb{Z}^2$ leads to a solution to $(\ref{MyPhaseLags})$. The perturbation system in the variables $\psi = \{\psi_{i,j}\}_{(i,j)\in\mathbb{Z}^2}$ then becomes
\begin{equation}
	\dot{\psi}_{i,j} = \sum_{i',j'} \sin(\psi_{i',j'} - \psi_{i,j})
\end{equation} 
for all $(i,j)\in\mathbb{Z}^2$. Linearizing this system of ordinary differential equations about the steady-state $\psi = 0$ results in the linear operator, denoted $L_1$, acting on the sequences $x = \{x_{i,j}\}_{(i,j) \in \mathbb{Z}^2}$ by 
\begin{equation}
	[L_1x]_{i,j} = \sum_{i',j'} (x_{i',j'} - x_{i,j}),
\end{equation}
for every $(i,j)\in\mathbb{Z}^2$, which is exactly the linearization (\ref{LinearizationEx}) discussed in Section~\ref{sec:Perturbation}. Here we interpret the underlying graph, denoted $G_1 = (\mathbb{Z}^2,E_1,w_1)$, to have vertex set $\mathbb{Z}^2$ and an edge set $E_1$ containing all nearest-neighbour interactions between vertices. The weights of each edge are identically given by $1$, thus giving that the measure of each vertex is identically $4$. An illustration of $G_1$ is given in Figure $\ref{fig:TrivialGraph}$ for visual reference. 

\begin{figure} 
	\centering
	\includegraphics[width = 7cm]{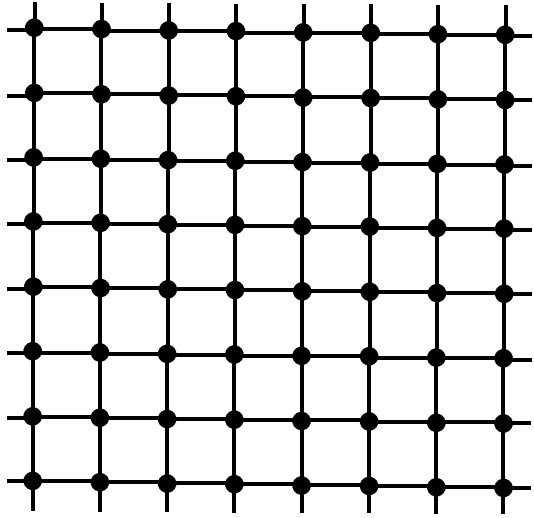}
	\caption{A representation of the graph $G_1$ associated with the linearization of $(\ref{MyPhaseSystem1})$ about the trivial phase-locked solution. Dots represent vertices of the graph and the lines connecting these vertices represent the edges.}
	\label{fig:TrivialGraph}
\end{figure}  

We see that since each vertex in $G_1$ has measure exactly $4$, we are in the first case of Hypothesis $\ref{Hyp:DecayRates}$. Following the discussion prior to Hypothesis $\ref{Hyp:DecayRates}$, we apply the linear time re-parametrization $t \to 4t$ to system $(\ref{PerturbationSystem2})$. Our re-parametrized system now results in the linearization about $\psi = 0$ given by
\begin{equation}
	[\tilde{L}_1x_{i,j}] = \sum_{i',j'} \frac{1}{4}(x_{i',j'} - x_{i,j}),
\end{equation}      
for every $(i,j)\in\mathbb{Z}^2$. Due to the fact that the weight of each vertex is identical, we have not added any loops to the original underlying graph $G_1$. What is important to note though is that now one sees that $-\tilde{L}_1$ is in the form of a graph Laplacian.

The operator $\tilde{L}_1$ and its underlying graph $G_1$ are a well-studied example in the theory of random walks on infinite graphs. Following Definition $\ref{def:VolumeGrowth}$ it was pointed out that this graph $G_1$ satisfies $VD(2)$, and Telcs points out in the proof of Theorem $3$ of $\cite{Telcs}$ that this graph satisfies the properties $PI$ and $\Delta$ as well. Therefore, Hypothesis $\ref{Hyp:DecayRates}$ holds for the trivial phase-locked solution, along with Hypotheses $\ref{Hyp:Neighbours}$, $\ref{Hyp:Linearization}$ and $\ref{Hyp:GraphDistance}$, thus allowing one to apply Theorem~\ref{thm:Stability} to this situation. We state this here as a corollary of Theorem~\ref{thm:Stability}. 

\begin{cor}
	The trivial phase-locked solution to $(\ref{MyPhaseSystem1})$ given by $\{\omega t\}_{(i,j) \in \mathbb{Z}^2}$ is locally asymptotically stable with respect to small perturbations in $\ell^1(\mathbb{Z}^2)$.	
\end{cor}

\subsection{Stability of the Rotating Wave Solution}

We now turn our attention to a nontrivial solution of $(\ref{MyPhaseSystem1})$. It was shown in $\cite{MyWork}$ that there exists a solution resembling a rotating wave in that the values of the phase-lags about concentric rings of cells will increase from $0$ up to $2\pi$. Let us simply denote this solution by $\bar{\theta} = \{\bar{\theta}_{i,j}\}_{(i,j) \in \mathbb{Z}^2}$. For visual reference, the solution is illustrated in Figure $\ref{fig:RotWave}$. Here the core of the rotating wave is given by the four cell ring with indices $(i,j) = (0,0), (0,1), (1,0), (1,1)$, represented at the centre of Figure $\ref{fig:RotWave}$. 

\begin{figure} 
	\centering
	\includegraphics[width = 8cm]{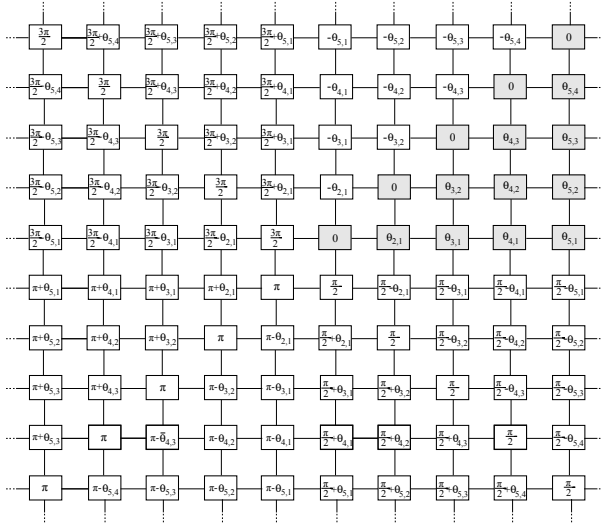}
	\caption{Symmetry of the phase-locked solution on the finite lattice. All phase-lags are obtained by phase advances and delays of the elements in the shaded cells. Image originally appears in $\cite{MyWork}$.}
	\label{fig:RotWave}
\end{figure} 

Linearizing the perturbation system about this solution leads to the linear operator, denoted $L_2$, acting upon the sequences $x = \{x_{i,j}\}_{(i,j) \in \mathbb{Z}^2}$ by
\begin{equation}
	[L_2x]_{i,j} = \sum_{i',j} \cos(\bar{\theta}_{i',j'} - \bar{\theta}_{i,j}) (x_{i',j'} - x_{i,j}), 
\end{equation}
for all $(i,j)\in\mathbb{Z}^2$. We again have that the conditions of Hypothesis $\ref{Hyp:Linearization}$ are satisfied since it was shown that all local interactions are such that $|\bar{\theta}_{i',j'} - \bar{\theta}_{i,j}| \leq \frac{\pi}{2}$. More precisely, with the exception of the 'centre' four cells at $(i,j) = (0,0), (0,1), (1,0), (1,1)$ all local interactions are such that $|\bar{\theta}_{i',j'} - \bar{\theta}_{i,j}| < \frac{\pi}{2}$, whereas the coupling between any two of the four centre cells is exactly $\pi/2$. Hence, we consider the weighted and connected graph, $G_2 = (\mathbb{Z}^2,E_2,w_2)$, with vertex set $\mathbb{Z}^2$ and an edge set which contains all nearest-neighbour connections less those connections between any two of the centre four cells since $\cos(\frac{\pi}{2}) = 0$. A visual representation of this graph is given in Figure $\ref{fig:RotGraph}$, where one notes the absence of connection between the core indices.    

We now remind the reader of some important properties of this rotating wave solution which were detailed in $\cite{MyWork}$. First, the solution is obtained via phase advances and phase delays of a solution obtained on the indices $1 \leq j \leq i$, which is represented by the shaded cells in Figure $\ref{fig:RotWave}$. Furthermore, we have that
\begin{equation}
	\begin{split}
	\begin{aligned}
		&\bar{\theta}_{i,i} = 0, \\
		&\bar{\theta}_{i,0} = \frac{\pi}{2} - \bar{\theta}_{i,1},
	\end{aligned}
	\end{split}
\end{equation} 
for all $i \geq 1$ and 
\begin{equation}
	0 < \bar{\theta}_{i,j} \leq \frac{\pi}{4}
\end{equation}
for all $1 \leq j < i$. Then one uses these facts to see that
\begin{equation} \label{NNBound1}
	|\bar{\theta}_{i',j'} - \bar{\theta}_{i,j}| \leq \frac{\pi}{4},
\end{equation} 
for all $1 \leq j \leq i$ and $1 \leq j' \leq i'$. Another property of the solution is that 
\begin{equation}
	\bar{\theta}_{i,j} \leq \bar{\theta}_{i+1,j}
\end{equation} 
for all $i \leq j \leq i$. One now sees that 
\begin{equation} \label{NNBound2}
	\frac{\pi}{2} > \bar{\theta}_{2,0} - \bar{\theta}_{2,1} = \frac{\pi}{2} - 2\bar{\theta}_{2,1} \geq \frac{\pi}{2} - 2\bar{\theta}_{i,1} = \bar{\theta}_{i,0} - \bar{\theta}_{i,1} \geq 0, 
\end{equation}
for all $i \geq 2$. Equations $(\ref{NNBound1})$ and $(\ref{NNBound2})$ therefore combine to show that all nearest-neighbour interactions within the indices $1 \leq j \leq i$ remain bounded away from $\pm\frac{\pi}{2}$. This therefore gives that
\begin{equation} \label{NNEdges}
	0 < \inf_{1 \leq j \leq i} \cos(\bar{\theta}_{i',j'} - \bar{\theta}_{i,j}) \leq  \sup_{1 \leq j \leq i} \cos(\bar{\theta}_{i',j'} - \bar{\theta}_{i,j}) \leq 1.
\end{equation}  
Since the elements at the indices $1 \leq j \leq i$ are used to define the solution over all the indices, one has that the edge weights are uniformly bounded above and away from $0$.  

\begin{figure} 
	\centering
	\includegraphics[width = 7cm]{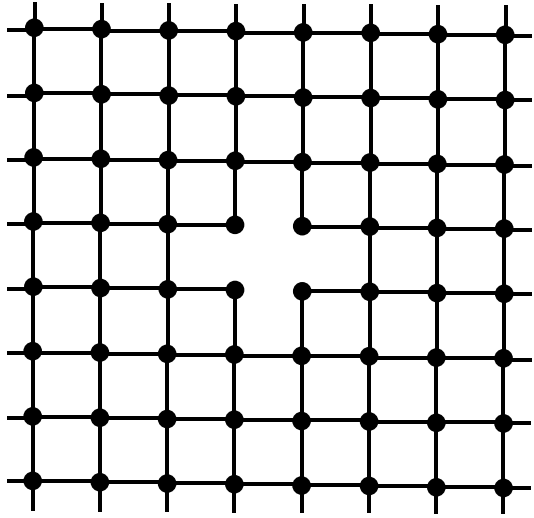}
	\caption{A representation of the graph $G_2$.}
	\label{fig:RotGraph}
\end{figure} 

Now one should note that $G_2$ is significantly different from $G_1$ in that not all vertices have the same number of edges attached to it and that the weights of each edge are not identical, thus putting us in the situation of the second case of Hypothesis $\ref{Hyp:DecayRates}$. Let $w_{min} > 0$ and $w_{max} > 0$ be uniform lower and upper bounds on the weight function, respectively. This allows one to apply the time re-parametrization given by $t \to (4w_{max}+1)t$ (since each vertex has degree at most $4$), resulting in the linear operator 
\begin{equation}
	\tilde{L}_2 := \frac{1}{(4w_{max}+1)}L_2
\end{equation}
and resulting graph $\tilde{G}_2 = (\mathbb{Z}^2,\tilde{E}_2,\tilde{w}_2)$. Recall from our work in Section $\ref{sec:Stability}$ that the graph $\tilde{G}_2$ is merely the graph $G_2$ with added edges connecting each vertex to itself (loops). Moreover, by the construction $(\ref{LoopConstruction})$, the weight of each loop is bounded above by $4w_{max}+1$ and below by $1$. Therefore the edge weights of $\tilde{G}_2$ are uniformly bounded above by an $\tilde{w}_{max} > 0$ and below by an $\tilde{w}_{min} > 0$. Then, Lemma $\ref{lem:Delta}$ implies that there exists an $\alpha > 0$ such that $\tilde{G}_2$ satisfies $\Delta$. To show that $\tilde{G}_2$ further satisfies $PI$ and $VD(2)$, we provide the following proposition.        

\begin{prop} \label{prop:RoughIsometry} 
	The graphs $G_1$ and $\tilde{G}_2$ are rough isomorphic. 	
\end{prop}

\begin{proof}
	Since $G_1$ and $\tilde{G}_2$ have the same vertex set, let us consider the identity mapping $\mathcal{I}: \mathbb{Z}^2 \to \mathbb{Z}^2$ which acts by $\mathcal{I}((i,j)) = (i,j)$ for all $(i,j) \in \mathbb{Z}^2$. We will systematically verify the three rough isometry properties $(\ref{Rough1})$, $(\ref{Rough2})$ and $(\ref{Rough3})$ to show that $\mathcal{I}$ is a rough isometry between the graphs $G_1$ and $\tilde{G}_2$. We will let $\rho_1$, $m_1$ be the distance and vertex weight functions on the graph $G_1$ and $\rho_2$, $m_2$ be the distance and vertex weight functions on the graph $\tilde{G}_2$.
	
	\underline{Property $(\ref{Rough1})$}: Recall that the edge sets of $G_1$ and $\tilde{G}_2$ differ only by the centre four edges which are present in the former and absent in the latter. Let $n_1 = (i_1,j_1), n_2 = (i_2,j_2) \in \mathbb{Z}^2$. Then since $G_1$ has more connections between vertices than $\tilde{G}_2$, one immediately has
	\begin{equation} \label{RoughProp1}
		\rho_1(n_1,n_2) \leq \rho_2(n_1,n_2),
	\end{equation}   
	since any path between the vertices $n_1$ and $n_2$ in $\tilde{G}_2$ could potentially be shortened by the addition of edges connecting different vertices. Conversely, any shortest path connecting vertices in $G_1$ potentially traverses an edge which is absent in $\tilde{G}_2$. Following along this path in $\tilde{G}_2$ requires one to replace those steps across the missing edges with two addition steps to circumvent the missing edge. Since there are a maximum of four missing edges to circumvent, we obtain 
	\begin{equation} \label{RoughProp2}
		\rho_2(n_1,n_2) \leq \rho_1(n_1,n_2) + 8.
	\end{equation}
	Therefore, to obtain the bounds $(\ref{Rough1})$ we use $(\ref{RoughProp1})$ and $(\ref{RoughProp2})$ and take, for example, $a = 2$ and $b = 8$ to have
	\begin{equation}
		\frac{1}{2}\rho_1(n_1,n_2) - 8 \leq \rho_2(n_1,n_2) \leq 2\rho_1(n_1,n_2) + 8.
	\end{equation}
	
	\underline{Property $(\ref{Rough2})$}: This property is trivially satisfied for any $M > 0$ since for all $n = (i,j)\in\mathbb{Z}^2$ we have $\rho_2(\mathbb{Z}^2,n) = 0$.
	
	\underline{Property $(\ref{Rough3})$}: Recall that for all $n = (i,j)\in\mathbb{Z}^2$ we have $m_1(n) = 4$. Furthermore, we keep with the notation above to denote $\tilde{w}_{min} > 0$ and $\tilde{w}_{max} > 0$ as uniform lower and upper bounds, respectively, on the weight function of $\tilde{G}_2$. Therefore, since each vertex has at least one edge connected to it and at most five (four nearest-neighbours and one loop) we have
	\begin{equation}
		\tilde{w}_{min} \leq m_2(n) \leq 5\tilde{w}_{max},
	\end{equation}
	for all $n = (i,j)\in\mathbb{Z}^2$. Hence, 
	\begin{equation}
		\frac{\tilde{w}_{min}}{4}m_1(n) = \tilde{w}_{min} \leq m_2(n) \leq 5\tilde{w}_{max} = \frac{5\tilde{w}_{max}}{4}m_1(n).
	\end{equation} 
	Taking $c := \max\{\frac{5\tilde{w}_{max}}{4},\frac{4}{\tilde{w}_{min}}, 2\} > 1$ gives
	\begin{equation}
		c^{-1}m_1(n) \leq m_2(n) \leq cm_1(n)
	\end{equation}
	for all $n = (i,j) \in \mathbb{Z}^2$. This completes the proof since we have shown that $\mathcal{I}:\mathbb{Z}^2 \to \mathbb{Z}^2$ satisfies all three conditions to be a rough isometry. 
\end{proof}

\begin{cor} \label{cor:RotWaveStability}
	$\tilde{G}_2$ satisfies $VG(2)$, $PI$ and $\Delta$.  
\end{cor}

\begin{proof}
	This is a direct consequence of Propositions $\ref{prop:RoughInvariance}$ and $\ref{prop:RoughIsometry}$.
\end{proof}

Therefore, Corollary $\ref{cor:RotWaveStability}$ allows one to now apply the stability results of Theorem~\ref{thm:Stability} since we have verified all of the necessary hypotheses. We summarize these results as a corollary of Theorem~\ref{thm:Stability} and the work in $\cite{MyWork}$.

\begin{cor}
	There exists a phases-locked rotating wave solution to $(\ref{MyPhaseSystem1})$ which is locally asymptotically stable with respect to small perturbations in $\ell^1(\mathbb{Z}^2)$. 
\end{cor}

\subsection{Doubly Spatially Periodic Stable Patterns} 

Now that we have two applications of Theorem~\ref{thm:Stability} to specific phase-locked solutions, let us demonstrate that system $(\ref{MyPhaseSystem1})$ can exhibit infinitely many stable phase-locked solutions. Begin by fixing any two integers $N_1,N_2 \geq 5$ and consider the phase-lags given by
\begin{equation} \label{Lags2}
	\bar{\theta}_{i,j} = \frac{2\pi [i]_{N_1}}{N_1} +  \frac{2\pi [j]_{N_2}}{N_2},
\end{equation}
where we have used the notation $[n]_N = n \pmod{N}$. These solutions form periodic waves in both the horizontal and vertical directions. It is easy to see that solutions of this type solve $(\ref{MyPhaseLags})$ since for all $(i,j)\in\mathbb{Z}^2$ we have
\begin{equation}
	\begin{split}
	\begin{aligned}
		\sin(\bar{\theta}_{i+1,j} - \bar{\theta}_{i,j}) + \sin(\bar{\theta}_{i-1,j} - \bar{\theta}_{i,j}) = \sin\bigg(\frac{2\pi}{N_1}\bigg) + \sin\bigg(-\frac{2\pi}{N_1}\bigg) = 0,  
	\end{aligned}
	\end{split}
\end{equation}
where we have used the $2\pi$-periodicity of sine and the fact that it is an odd function. An equivalent formulation holds for the up/down connections, giving that the left/right and up/down connections of $(\ref{MyPhaseSystem1})$ independently cancel to satisfy $(\ref{MyPhaseLags})$.

Applying the perturbation ansatz and linearizing about $\psi = 0$ results in a linear operator $L_3$ which acts on the sequences $x = \{x_{i,j}\}_{(i,j)\in\mathbb{Z}^2}$ by
\begin{equation}
	\begin{split}
		[L_3 x]_{i,j} = \cos\bigg(&\frac{2\pi}{N_1}\bigg)\bigg[(x_{i+1,j} - x_{i,j}) + (x_{i-1,j} - x_{i,j})\bigg] \\ 
		&+ \cos\bigg(\frac{2\pi}{N_2}\bigg)\bigg[(x_{i,j+1} - x_{i,j}) + (x_{i,j-1} - x_{i,j})\bigg].
	\end{split}
\end{equation}
Furthermore, since $N_1,N_2 \geq 5$ we have that 
\begin{equation}
	0 < \frac{2\pi}{N_1}, \frac{2\pi}{N_2} < \frac{\pi}{2} \implies \cos\bigg(\frac{2\pi}{N_1}\bigg), \cos\bigg(\frac{2\pi}{N_2}\bigg) > 0, 
\end{equation}
giving that the symmetry and positivity requirements of $(\ref{PositiveWeights})$ have been met. Here the underlying graph, $G_3 = (\mathbb{Z}^2,E_3,w_3)$, has vertex set $\mathbb{Z}^2$ and an edge set, $E_3$ which contains all nearest-neighbour interactions, similar to $G_1$. The difference now is that the associated weight function assigns a weight of $\cos(\frac{2\pi}{N_1})$ to horizontal (left/right) edges and a weight of $\cos(\frac{2\pi}{N_2})$ to vertical (up/down) edges. Clearly this graph is connected, and looks indistinguishable to that given in Figure $\ref{fig:TrivialGraph}$ for $G_1$ when edge weights are not labelled.

Now, the measure of each vertex is given identically by $2\cos(\frac{2\pi}{N_1}) + 2\cos(\frac{2\pi}{N_2})$. Hence, following the discussion prior to Hypothesis $(\ref{Hyp:DecayRates})$, we apply the linear time re-parametrization $t \to [2\cos(\frac{2\pi}{N_1}) + 2\cos(\frac{2\pi}{N_2})]t$ to system $(\ref{PerturbationSystem2})$. This results in a graph Laplacian upon linearizing about the phase-lags $(\ref{Lags2})$, denoted 
\begin{equation}
	\tilde{L}_3 = \frac{1}{2\cos(\frac{2\pi}{N_1}) + 2\cos(\frac{2\pi}{N_2})} L_3.
\end{equation} 
One should note in this case, as with the trivial solution, our re-parametrization of $t$ has not added any loops to the original underlying graph $G_3$. In fact, we have not altered the underlying graph $G_3$ in any way since no edges have been added.

\begin{prop} \label{prop:RoughIsometry2} 
	The graphs $G_1$ and $G_3$ are rough isomorphic. 
\end{prop}

\begin{proof}
	This proof follows in exactly the same was as the proof of Proposition $\ref{prop:RoughIsometry}$. We again consider the identity mapping $\mathcal{I}: \mathbb{Z}^2 \to \mathbb{Z}^2$ acting between the vertex sets of the respective graphs. Furthermore, since $G_1$ and $G_3$ have the same vertex and edge set, they have the same distance function. This immediately gives that property $(\ref{Rough1})$ is satisfied for any $a > 1$ and $b > 0$. 
	
	As previously stated, since the vertex sets of $G_1$ and $G_3$ are the same, property ($\ref{Rough2}$) is also immediately satisfied. This leaves one to check that property ($\ref{Rough3}$) is satisfied. But this is again obvious since the vertex weights in both $G_1$ and $G_3$ are independent of the vertices. Therefore, one finds that $G_1$ and $G_3$ are rough isomorphic.   
\end{proof}

\begin{cor}
	$G_3$ satisfies $VD(2)$, $PI$ and $\Delta$.
\end{cor}

Hence, we see that we have satisfied all of the hypotheses to apply Theorem~\ref{thm:Stability}. This leads to the following corollary. 

\begin{cor}
	For any two integers $N_1,N_2 \geq 5$, the phase-locked solution to $(\ref{MyPhaseSystem1})$ given by $\{\omega t + \bar{\theta}_{i,j}\}_{(i,j) \in \mathbb{Z}^2}$ with $\bar{\theta}_{i,j}$ as in $(\ref{Lags2})$, is locally asymptotically stable with respect to small perturbations in $\ell^1(\mathbb{Z}^2)$. 
\end{cor}

\section{Discussion} \label{sec:Discussion} 

In this work we have provided sufficient conditions to obtain local asymptotic stability of phase-locked solutions to infinite systems of coupled oscillators. We saw that this stability result was obtained from a unification of some results pertaining to random walks on infinite weighted graphs. It is clear that the results in this work are best suited for odd coupling functions, with the example examined in detail in the work being $H(x) = \sin(x)$. Moreover, such a sinusoidal coupling function is not a mere mathematical idealization to demonstrate the uses of this new stability result, but conversely our hypotheses and assumptions are based upon well-studied paradigms in the theory of weakly coupled oscillators, such as the Kuramoto model.

Aside from pointing out the symmetry limitation in Remark $\ref{rmk:Digraph}$, our work here has another minor shortcoming. One should notice that Hypothesis $\ref{Hyp:DecayRates}$ requires that the underlying graph satisfy $VG(d)$ for some $d \geq 2$, leaving the case when $d = 1$ out of the scope of Theorem~\ref{thm:Stability}. Although this appears as a minor detail in our hypotheses, the proof of Theorem~\ref{thm:Stability} heavily relies on the fact that $d\geq 2$ in order to properly apply Lemma $\ref{lem:IntegralLemma}$. It appears that the cases when $d = 1$ cannot be reconciled using the methods put forth in this paper, but it should be noted that techniques such as the discrete Fourier transform could potentially be applied to obtain similar results. Furthermore, our examples have used positive integer values for $d$ due to their relevance to the important integer lattices, but it should be noted that it is possible to construct fractal graphs which dimension $d$ which need not be an integer. Although it could be difficult to motivate the study of coupled oscillators on fractal graphs, it is certainly interesting to note that Theorem~\ref{thm:Stability} makes no claim that $d$ is an integer, and therefore could be applied to fractal graphs with dimension $d \geq 2$ so long as the remaining assumptions are satisfied.   

Let us examine a brief example of a case where Theorem~\ref{thm:Stability} cannot be applied. Consider the system with $V = \mathbb{Z}$ given by
\begin{equation} \label{ChainSystem}
	\dot{\theta}_i = \omega + \sin(\theta_{i+1} - \theta_i) + \sin(\theta_{i-1} - \theta_i),
\end{equation}   
for all $i \in \mathbb{Z}$ and some $\omega\in\mathbb{R}$. One may follow the same techniques laid out in Section $\ref{subsec:Trivial}$ to inspect the stability of the trivial phase-locked solution given by $\theta_i(t) = \omega t$, for all $i \in \mathbb{Z}$. We refrain from going into detail here, but one can show that the underlying graph is the one-dimensional version of $G_1$ from Section $\ref{subsec:Trivial}$, for which we give a visual representation in Figure $\ref{fig:LinearGraph}$. Furthermore, in the discussion following Definition $\ref{def:VolumeGrowth}$ it was pointed out that this graph satisfies $VG(1)$, thus we find that we are unable to determine stability of this trivial phase-locked solution using Theorem~\ref{thm:Stability}. Of course, stability of the linearized equation can be determined via the random walk probabilities outlined in this manuscript, or with the discrete Fourier transform in a manner similar to \cite{Stefanov}, but the extension to nonlinear stability cannot be obtained with the methods in this manuscript due to the relatively slow decay which prevents one from closing the bootstrapping argument.  

\begin{figure} 
	\centering
	\includegraphics[width = 7cm]{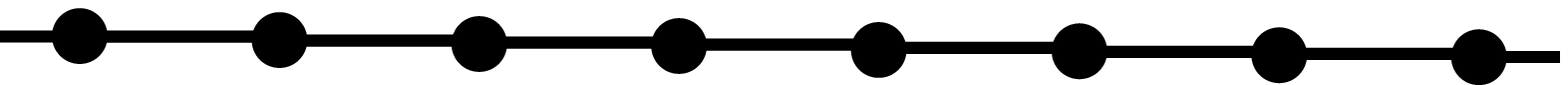}
	\caption{The graph corresponding to the trivial phase-locked solution to system $(\ref{ChainSystem})$. Here the vertices lie in one-to-one correspondence with the elements of $\mathbb{Z}$. All edges have weight exactly $1$ and connect the vertex at $i \in \mathbb{Z}$ to the vertices at $i\pm1$.}
	\label{fig:LinearGraph}
\end{figure} 

Most interestingly, it seems that this problem in system $(\ref{ChainSystem})$ cannot be reconciled by increasing the distance of influence for each oscillator. That is, consider a generalization of system $(\ref{ChainSystem})$ given by the phase system
\begin{equation} \label{GeneralChainSystem}
	\dot{\theta}_i = \omega + \sum_{j = i-n}^{i+n} \sin(\theta_j - \theta_i),
\end{equation}
for each $i \in \mathbb{Z}$ and a fixed $n \geq 1$. One finds that the trivial phase-locked solution again relates to a weighted graph which satisfies $VG(1)$ (for example, see Figure $\ref{fig:LinearGraph2}$), thus showing that by increasing the range of influences we still cannot conclude local asymptotic stability. Hence, one hopes to determine an appropriate analogue of Theorem~\ref{thm:Stability} which works for any graph satisfying $VG(d)$ for arbitrary $d > 0$.    

\begin{figure} 
	\centering
	\includegraphics[width = 7cm]{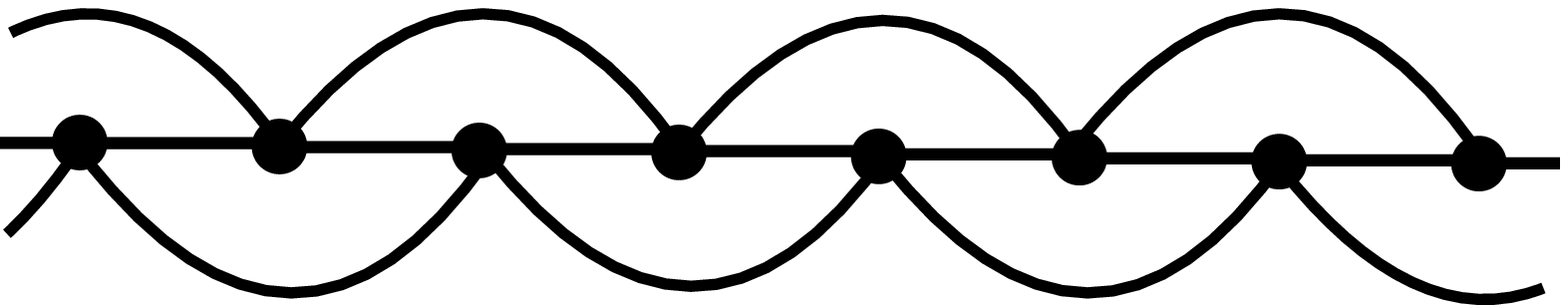}
	\caption{The graph corresponding to the trivial phase-locked solution to system $(\ref{GeneralChainSystem})$ with $n = 2$. Here again the vertices lie in one-to-one correspondence with the elements of $\mathbb{Z}$. All edges have weight exactly $1$ and connect the vertex at $i \in \mathbb{Z}$ to the vertices at $i\pm1$ and $i\pm2$.}
	\label{fig:LinearGraph2}
\end{figure} 

The work in this paper is only a starting point for many necessary further investigations into the study of infinitely many coupled oscillators. Now that a link between the stability of phase-locked solutions and random walks on weighted graphs has been established, it is the hope that questions in one area could motivate research in the other. It is also the hope that many of the results presented in Section $\ref{sec:Graphs}$ for undirected weighted graphs can be extended in an appropriate way for weighted digraphs, which would in turn help to provide a more robust stability result than that which has been established in Theorem~\ref{thm:Stability}.

\section*{Acknowledgements} 

This work is supported by an Ontario Graduate Scholarship while at the University of Ottawa. The author is very thankful to Benoit Dionne and Victor LeBlanc for their comments on how to properly convey the results. The author would also like to acknowledge a beneficial correspondence with Erik Van Vleck of the University of Kansas which initiated the work in this manuscript.

\end{document}